\documentclass[11pt,fleqn]{amsart} 

\usepackage[usenames,dvipsnames,condensed]{xcolor}

\usepackage{amsthm}
\usepackage{amsfonts}
\usepackage[english]{babel}
\usepackage[usenames]{xcolor}
\usepackage{graphicx}
\usepackage{soul}
\usepackage{stfloats}
\usepackage{morefloats}
\usepackage{cite}
\usepackage{lscape}
\usepackage{epstopdf}
\usepackage{braket}
\usepackage[lite]{amsrefs}
\usepackage{mathbbol}
\usepackage{tikz}

\usepackage{amsmath,amstext,amsopn,amsfonts,eucal,amssymb}
\usepackage{graphicx,wrapfig,url}

\newcommand\Z{{\mathbb Z}}

\newtheorem{theorem}{Theorem}[section]

\newtheorem{lemma}[theorem]{Lemma}
\newtheorem{proposition}[theorem]{Proposition}
\newtheorem{definition}[theorem]{Definition}

\newtheorem{remark}[theorem]{Remark}

\newcommand{\cal}{\mathcal}

 \usepackage{overpic}

\DeclareMathOperator*{\Motimes}{\text{\raisebox{0.25ex}{\scalebox{0.7}{$\bigotimes$}}}}

\begin{document}

\title{Yang-Baxter Solutions from Categorical Augmented Racks} 

\author{Masahico Saito} 
\address{Department of Mathematics and Statistics, 
	University of South Florida, Tampa, FL 33620, U.S.A.} 
\email{saito@usf.edu} 

\author{Emanuele Zappala} 
\address{Department of Mathematics and Statistics, Idaho State University\\
	Physical Science Complex |  921 S. 8th Ave., Stop 8085 | Pocatello, ID 83209} 
\email{emanuelezappala@isu.edu}

\begin{abstract}
An augmented rack is a set with a self-distributive binary operation induced by a group action, and has been extensively used in knot theory. Solutions to the Yang-Baxter equation (YBE)
have been also used for knots, since the discovery of the Jones polynomial. In this paper, an interpretation of augmented racks in tensor categories is given for coalgebras that are Hopf algebra modules, and associated solutions to the YBE are constructed. 
Explicit constructions are given using quantum heaps and the adjoint of Hopf algebras. 
Furthermore,  an inductive construction of Yang-Baxter solutions is given by means of the categorical augmented racks, yielding infinite families of solutions.
Constructions of braided monoidal categories are also provided using categorical augmented racks.
\end{abstract}

\maketitle

\date{\empty}

\tableofcontents

\section{Introduction}

The {\it Yang-Baxter equation} (YBE) has played a significant role in physics and knot theory.
In knot theory, since the discovery of the Jones polynomial, its relation to YBE lead to wide range
of knot invariants, opening a new era of quantum topology.
Constructing solutions to YBE has been an important branch in this development.
Theory of quantum groups \cite{ChariPressley} played a key role here.
Specifically, for a unital ring ${\mathbb k}$, a ${\mathbb k}$-module $X$
and an  invertible map $R: X \otimes X \rightarrow X \otimes X$,
the YBE is defined by 
$$(R \otimes {\mathbb 1}) ({\mathbb 1} \otimes R) (R \otimes {\mathbb 1}) 
=({\mathbb 1} \otimes R) (R \otimes {\mathbb 1}) ({\mathbb 1} \otimes R) ,$$
where ${\mathbb 1}$ denotes the identity map. 
A solution to YBE is also called a Yang-Baxter (YB) operator or R-matrix.
A diagrammatic representation of YBE is depicted in Figure~\ref{braiding}.
Diagrams are read from top to bottom. A straight line represents the identity ${\mathbb 1}$, each  crossing represents the map $R$, and the horizontal juxtaposition represents the tensor product of maps, 
so that the top left crossing with a straight line to its right represents $R \otimes {\mathbb 1}$. 
These diagrams also represent a braid relation 
and the Reidemeister type III move.

\begin{figure}[htb]
\begin{center}
\begin{overpic}[width=1.5in]{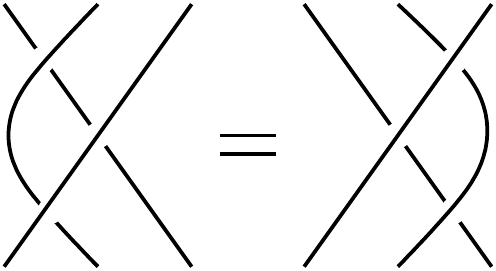}
	\put(-35,38){$R \otimes {\mathbb 1}$}
		\put(-35,23){$  {\mathbb 1} \otimes R $}
			\put(-35,8){$R \otimes {\mathbb 1}$}
	\put(109,38){$  {\mathbb 1} \otimes R $}		
	\put(109,23){$R \otimes {\mathbb 1}$}	
	\put(109,8){$  {\mathbb 1} \otimes R $}
\end{overpic}
\end{center}
\caption{Diagrams for the Yang-Baxter equation}
\label{braiding}
\end{figure}

The discrete case of the set-theoretic YBE also played a key role in knot theory.
Self-distributive operations called racks and quandles \cites{Joyce82,Matveev82} and their cohomology theories
\cites{FR,CJKLS} have been extensively studied for constructions of invariants for knots and knotted surfaces, producing significant novel applications in knot theory.
A rack operation can be used to construct solutions to the set-theoretic YBE, that gives
representations of braid groups.
On the other hand, self-distributive structures in  coalgebra category have been 
formulated in \cite{CCES1}. 
Thus it is a natural attempt to apply this principle of constructing solutions of set-theoretic YBE by racks to  coalgebra categories.
This is done in the present article.

An important family or racks is represented by {\it augmented} racks, describing the operation in terms of group actions that are related to conjugation (see below). 
Although an interpretation of augmented rack in a tensor category was mentioned in \cite{ESZheap}
in relation to a ternary self-distributive operation called {\it heap}, 
a  definition for
binary case in tensor categories 
and its applications have not been fully explored. 

In this paper we  provide a definition of  augmented racks
for coalgebras, 
and give constructions of solutions to YBE using them.
Three explicit examples are presented, using quantum heaps, adjoint of Hopf algebras, and doubling constructions.
Thus for these examples, novel constructions of YBE solutions are provided. 
Further generalizing the doubling constructions, we also provide a method of constructing a braided monoidal category using a family of YBE solutions resulting from  the categorical augmented rack structures.

Below we describe the organization  and further details of results of the paper.
Reviews of definitions and terminologies used in this paper are given in Section~\ref{sec:prelim}.
The definition of categorical augmented racks for coalagebras that are modules over Hopf algebras  is presented in Section~\ref{sec:CAR}. Using such a structure, self-distributive maps are constructed, and they are further utilized to 
obtain solutions to the YBE. 
In the  two sections  that follow,  explicit examples to which our approach can be applied are presented.
In Section~\ref{sec:heap}, these constructions are applied to the Hopf algebra version of a group heap.
In Section~\ref{sec:ad} the adjoint of Hopf algebras is used for a family of examples.
An inductive construction is provided in Section~\ref{sec:doubling}, that provides infinite solutions of YBE from one.
This inductive construction is strengthened  in Section~\ref{sec:BMC} to derive a braided monoidal category 
generated by one object, providing such categories from the explicit examples established in earlier sections.

\section{Preliminaries}\label{sec:prelim}

\subsection{Augmented racks}

A {\it rack}~\cite{FR} is a set $X$ with a binary operation $*$ such that 
$R_y: X \rightarrow X$ defined by $x*y=R_y(x)$ is an automorphism of $*$ for all $y \in X$.
This is equivalent to the conditions that $R_y$ is a bijection for all $y \in X$ and $*$ is (right) self-distributive (SD),
$$(x*y)*z=R_z (x*y)=R_z(x) * R_z(y)=(x*z)*(y*z)$$ for all $x,y,z \in X$.
A {\it quandle} is a rack $X$ that satisfies the idempotency, $x*x=x$ for all $x \in X$.

Typical examples of quandles  include the trivial quandle $x*y=x$ on any set, $n$-fold conjugations $x*y=y^{-n} x y$ on  groups,
the {\it core} quandle $x*y=y x^{-1} y$ on  groups, and {\it Alexander quandle}
$x*y=tx + (1-t)y$ on $\Z[t, t^{-1}]$ modules.
A {\it cyclic} rack $x*y=x+1$ on $\Z_n$ is a rack that is not a quandle. 

An {\it augmented rack}~\cite{FR} $(X, G)$ is a set $X$ with a right group action by  a group $G$ 
and a map $\nu: X \rightarrow G$ satisfying the identity
$\nu (x \cdot g) = g^{-1} \nu(x) g$ for all $x \in X$, $g \in G$.
An augmented rack has a rack operation defined by $x*y=x \cdot  \nu (y)$ for $x, y \in X$.
A group $G$ with the inner automorphism $\nu={\rm id}_G$, $x*y=x\cdot \nu(y)=y^{-1} x y$, $x,y \in G$,
gives an augmented rack.

A {\it homomorphism of augmented racks} $(X_1,G, \nu_1)$ and $(X_2,G, \nu_2)$ is a $G$-equivariant map $f: X_1 \rightarrow X_2$, i.e. $f(x\cdot g) = f(x)\cdot g$, satisfying the property that 
$\nu_2(f(x)) = \nu_1(x)$. An augmented rack homomorphism which is invertible through an augmented rack homomorphism is said to be an isomorphism.

\subsection{Hopf algebras}

A {\it Hopf algebra} $(X, \mu,  \eta,  \Delta, \epsilon, S)$ (a 
module over a unital ring 
$\mathbb k$,
multiplication, unit, comultiplication, counit, antipode,  respectively), is
defined as follows. 
First, recall that a bialgebra 
$X$  
is a module endowed with
a multiplication $\mu: X\otimes X\rightarrow X$ with unit $\eta$ and a comultiplication $\Delta: X\rightarrow X\otimes X$ with counit $\epsilon$ such that the compatibility condition 
$$\Delta \circ  \mu = (\mu\otimes \mu)\circ ( {\mathbb 1}\otimes  \tau\otimes  {\mathbb 1}) \circ (\Delta\otimes \Delta)$$
holds,
where $\tau$ denotes the transposition $\tau(x \otimes y) = y\otimes x$ for simple tensors.  A Hopf algebra is a bialgebra endowed with a map $S: X\rightarrow X$, called {\it antipode}, satisfying the equations 
$$\mu \circ (\mathbb 1\otimes S)\circ \Delta = \eta \circ \epsilon = \mu\circ (S\otimes \mathbb 1)\circ\Delta,$$
 called the {\it antipode condition}.

The diagrammatic representation of the algebraic operations appearing in a Hopf algebra is given in Figure~\ref{hopfop}. 
The top two arcs 
of the trivalent vertex for multiplication $\mu$ (the leftmost diagram) represent $X \otimes X$, the vertex represents $\mu$, and the bottom arc represents $X$. 
In Figure~\ref{hopfaxiom} some of the defining axioms of a Hopf algebra are translated into diagrammatic equalities. Specifically, 
diagrams 
represent
(A) associativity of $\mu$, (B) unit condition, (C), compatibility between $\mu$ and $\Delta$, (D) the antipode condition. The coassociativity and  counit conditions are represented by diagrams that are vertical mirrors of (A) and (B), respectively. 

\begin{figure}[htb]
\begin{center}
\includegraphics[width=2.2in]{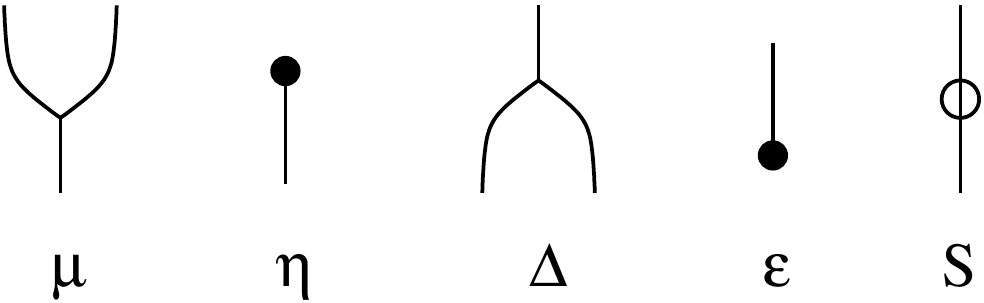}
\end{center}
\caption{Operations of Hopf algebras}
\label{hopfop}
\end{figure}

\begin{figure}[htb]
\begin{center}
\includegraphics[width=4.6in]{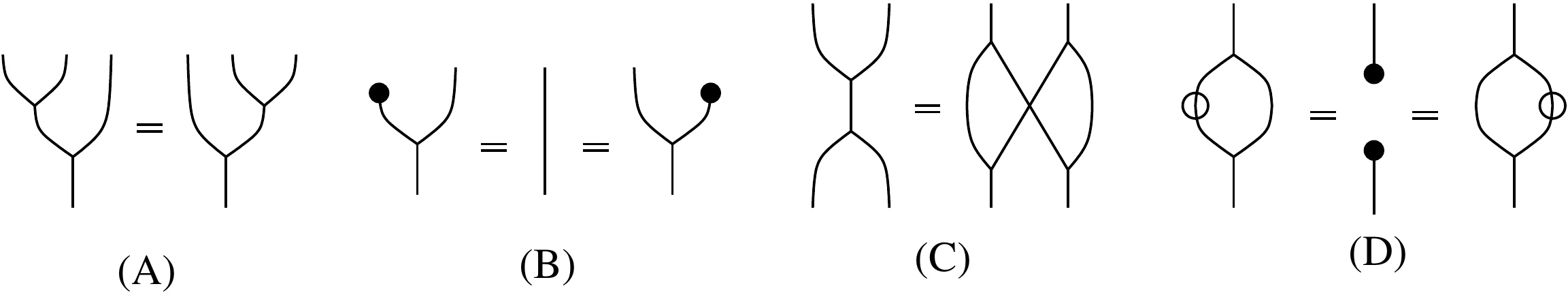}
\end{center}
\caption{Axioms of Hopf algebras}
\label{hopfaxiom}
\end{figure}

Any Hopf algebra satisfies the equality $\mu (S \otimes S) = S \mu  \tau $, where $\tau$ denotes the transposition $\tau(x \otimes y) = y\otimes x$ for simple tensors. 
This equality is depicted in Figure~\ref{mutwist}.
A similar equality holds for comultiplication, $ (S \otimes S) \Delta = \Delta S \tau$, represented by an upside down diagram.
A Hopf algebra is called {\it involutory} if $S^2={\mathbb 1}$, the identity. 


For the comultiplication, we use Sweedler's notation $\Delta(x)=x^{(1)}\otimes x^{(2)}$ suppressing the summation. Further, we use
$( \Delta \otimes {\mathbb 1} ) \Delta (x) = ( x^{(11)} \otimes x^{(12)} ) \otimes x^{(2)}$ and
$( {\mathbb 1}  \otimes \Delta ) \Delta (x) =  x^{(1)} \otimes ( x^{(21)} \otimes x^{(22)} ) $,
both of which are also written as 
$ x^{(1)} \otimes x^{(2)}  \otimes x^{(3)}$ from the coassociativity. 

\begin{figure}[htb]
\begin{center}
\includegraphics[width=1in]{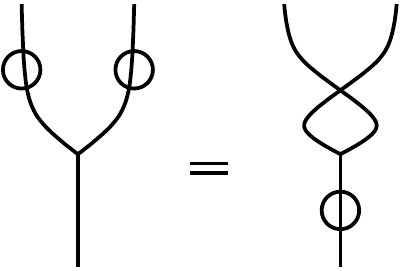}
\end{center}
\caption{Twisting $\mu$ with antipodes}
\label{mutwist}
\end{figure}

\section{Categorical augmented racks}\label{sec:CAR}

In this section we define categorical augmented racks, show that they can be equipped with categorical self-distributive maps, and provide associated R-matrices.

\begin{definition}\label{def:compati}
{\rm
Let $H$ be a Hopf algebra, and $X$ be a coalgebra that is a right $H$-module.
The comultiplication $\Delta_X$ of $X$ is said to be {\it compatible} with the module action 
$\cdot: X \otimes H \rightarrow X$ if they satisfy 
$\Delta_X ( x \cdot g)= (x^{(1)} \cdot g^{(1)} ) \otimes (x^{(2)} \cdot g^{(2)} ) $.
We also say, in this case, that the action is a coalgebra morphism.
}
\end{definition}

This condition is depicted in Figure~\ref{compati}. 
In the figure, thin black lines represent $H$ and thick blue lines represent $X$. 
The trivalent vertex where a black line merges to a blue line indicates the action.

\begin{figure}[htb]
\begin{center}
\includegraphics[width=1.3in]{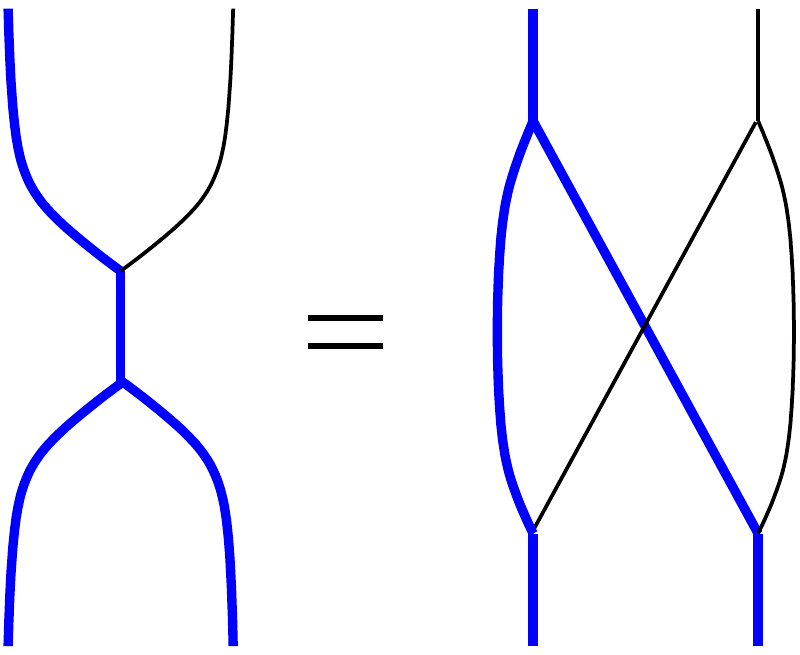}
\end{center}
\caption{Compatibility between action and comultiplication}
\label{compati}
\end{figure}

\begin{definition}\label{def:comm}
{\rm
Let $H$ be a Hopf algebra, and $X$ be a coalgebra that is a right $H$-module.
A linear map $f: X \rightarrow H$ is called {\it coalgebra  morphism} if 
it satisfies $\Delta_H f(x)= (f \otimes f)(\Delta_X (x))$ for all $x \in X$. 
}
\end{definition}

This condition is depicted in Figure~\ref{nucomult}, where the upside down triangle represents 
$f$. Later we will exclusively use the triangle notation for the augmentation map described below.

\begin{figure}[htb]
\begin{center}
\includegraphics[width=1.3in]{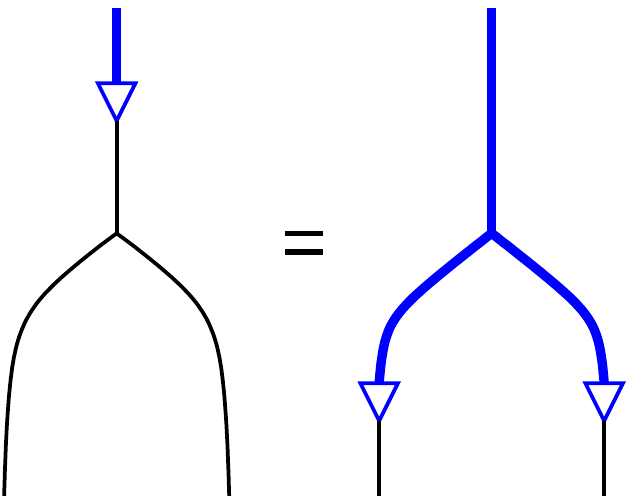}
\end{center}
\caption{Commutativity between $\nu$ and comultiplication}
\label{nucomult}
\end{figure}

\begin{definition}
{\rm
Let $H$ be a Hopf algebra, and $X$ be a coalgebra that is a right $H$-module, such that 
the action is compatible with the comultiplication $\Delta_X$ of $X$,
as diagrammatically depicted in Figure~\ref{compati}. An {\it augmentation map} $\nu: X \rightarrow H$ is a coalgebra morphism that satisfies
$\nu (x \cdot g)=S(g^{(1)} )\nu (x) g^{(2)}$ for all $x \in X$ and $g \in H$, where the juxtaposition on the right hand side represents the multiplication.
The diagrams for the left hand side and right hand side, respectively, 
of the above equality are depicted in Figure~\ref{aug} $(A)$ and $(B)$.
If there is an augmentation map $\nu$, $X$ is called {\it a categorical augmented rack}.
}
\end{definition}

Although the naming with {\it categorical } might suggest being defined for a general category,
we focus on coalgebra categories.

\begin{definition}
	{\rm 
			A homomorphism of categorical augmented racks $(X_1, H, \nu_1)$ and $(X_2, H, \nu_2)$ is an $H$-equivariant coalgebra morphism $f : X_1 \rightarrow X_2$, i.e. $f(x\cdot h_1) = f(x)\cdot h_2$, satisfying the property that 
			$\nu_2(f(x)) = \nu_1(x)$. A homomorphism that is invertible through a homomorphism of categorical augmented racks is an isomorphism. 
	}
\end{definition}

\begin{definition}[\cite{CCEScoalg}]
{\rm
Let $X$ be a coalgebra over ${\mathbb k}$.
A map $q: X \otimes X \rightarrow X$ is called {\it (categorically) self-distributive} (SD) if it satisfies
$$ q ( q \otimes {\mathbb 1} ) = 
q ( q \otimes q) ( {\mathbb 1} \otimes \tau \otimes  {\mathbb 1} ) ({\mathbb 1} \otimes {\mathbb 1} \otimes \Delta)
: X^{\otimes 3} \rightarrow X 
$$
where $\tau$ represents transposition on monomials as before, $\tau(x \otimes y ) = y \otimes x$.
}
\end{definition}

 Two SD maps $(X_1, q_1)$ and $(X_2, q_2)$ are called {\it equiovalent}
if there is a coalgebra homomorphism $f: X_1 \rightarrow X_2$ such that 
$q_2 f = f^{\otimes 3} q_1$. 

\begin{figure}[htb]
\begin{center}
\includegraphics[width=1.7in]{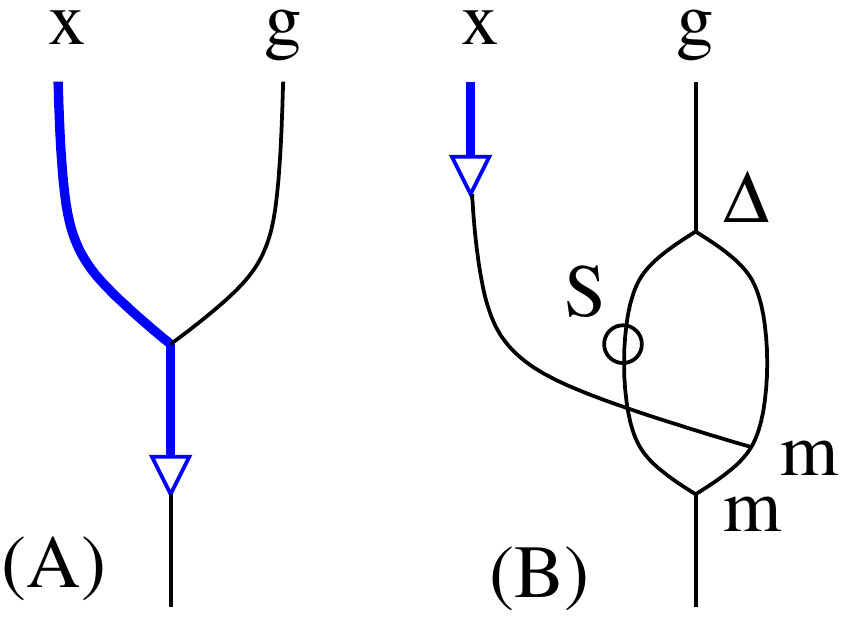}
\end{center}
\caption{Diagrams representing the definition of categorical augmented racks}
\label{aug}
\end{figure}

\begin{definition} \label{def:SD}
{\rm
Let $X$ be a categorical augmented rack over a Hopf algebra $H$ with the augmentation map $\nu$.
The  map $q: X \otimes X \rightarrow X$ defined by $q(x \otimes y)=x \cdot \nu(y)$
is called the self-distributive (SD) map associated with $\nu$.
}
\end{definition}

The left hand side and the right hand side of the above definition are diagrammatically represented by 
the top left and the bottom left of Figure~\ref{SD}, where a circled trivalent vertex represents 
the map $q$ in question.

\begin{lemma} 
The SD map $q$ associated with $\nu$ defined in Definition~\ref{def:SD} is indeed (categorically) self-distributive.
\end{lemma}

\begin{proof}
The proof is represented by a sequence of diagrams in Figure~\ref{SD}.
Specifically, the equality (1) is the definition of the SD map and  (2) is a property of an action.
The equality (3) is defining properties of a unit and a counit (Figure~\ref{hopfaxiom} (B)), 
(4) is the commutativity of transposition map and unit, and (5) is another axiom of Hopf algebras (Figure~\ref{hopfaxiom} (D)).
After commutativity between multiplication and transposition (6), (co)associativities are applied in (7), action axiom in (8), and the definitions are applied in (9) and (10). 
\end{proof}

\begin{figure}[htb]
\begin{center}
\includegraphics[width=4in]{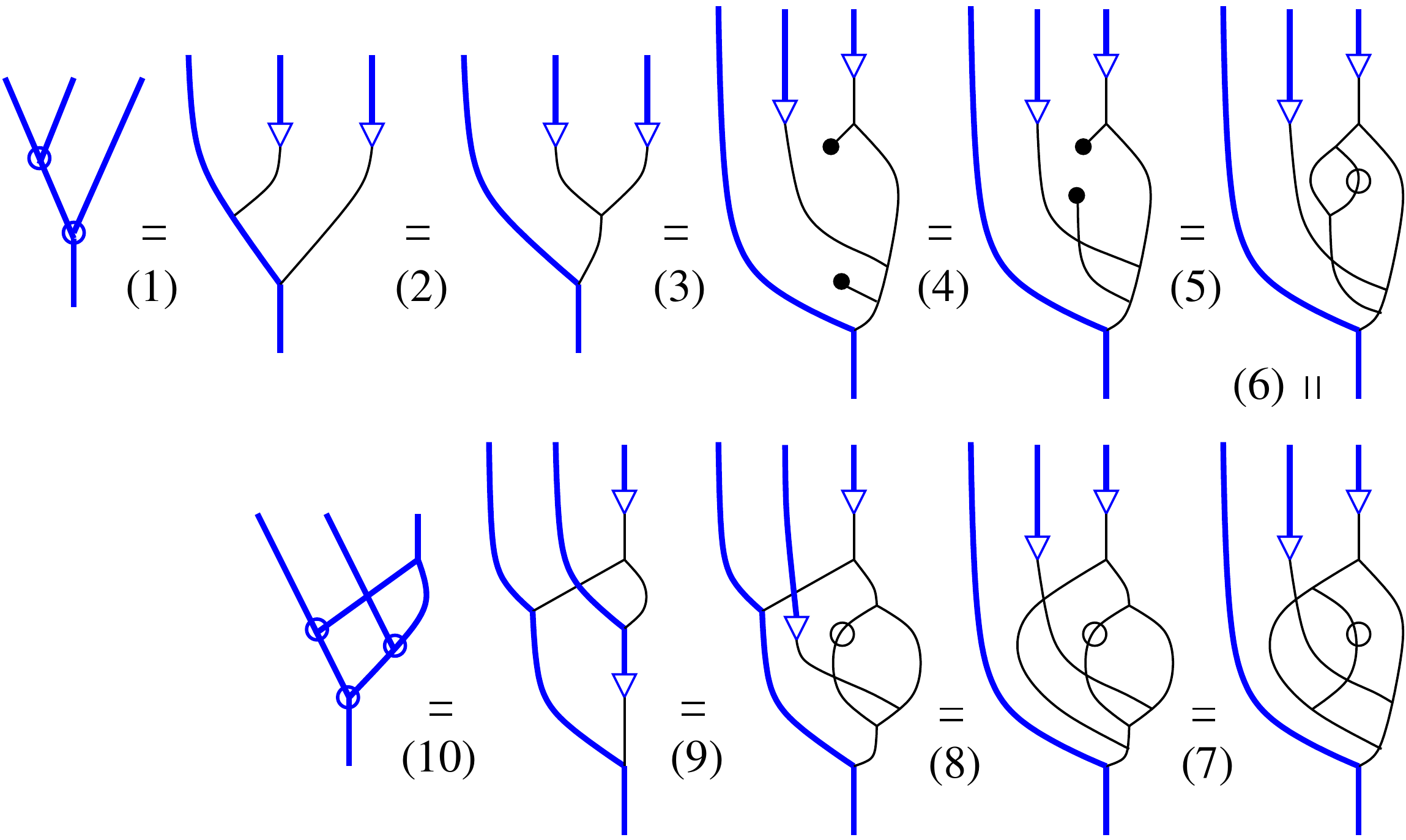}
\end{center}
\caption{Self-distributivity of categorical augmented racks}
\label{SD}
\end{figure}

\begin{remark}
	{\rm 
		Observe that a homomorphism of categorical augmented racks induces a homomorphism of SD structures, as a direct computation shows. If the homomorphism is an isomorphism of categorical augmented racks, then we obtain an isomorphism between SD structures.
	}
\end{remark}

\begin{lemma} \label{lem:SDcompati}
The SD map associated with $\nu$ defined in Definition~\ref{def:SD} is compatible with the comultiplication $\Delta_X$.\end{lemma}

\begin{proof}
The proof is represented by a sequence of diagrams in Figure~\ref{nucompati}.
Specifically, the equalities are consequences of: (1) the definition, (2)  the 
compatibility between action and comultiplication (Figure~\ref{compati}), 
(3)  the compatibility between $\nu$ and comultiplication (Figure~\ref{nucomult}), 
and (4) the definition. 
\end{proof}

\begin{figure}[htb]
\begin{center}
\includegraphics[width=3.5in]{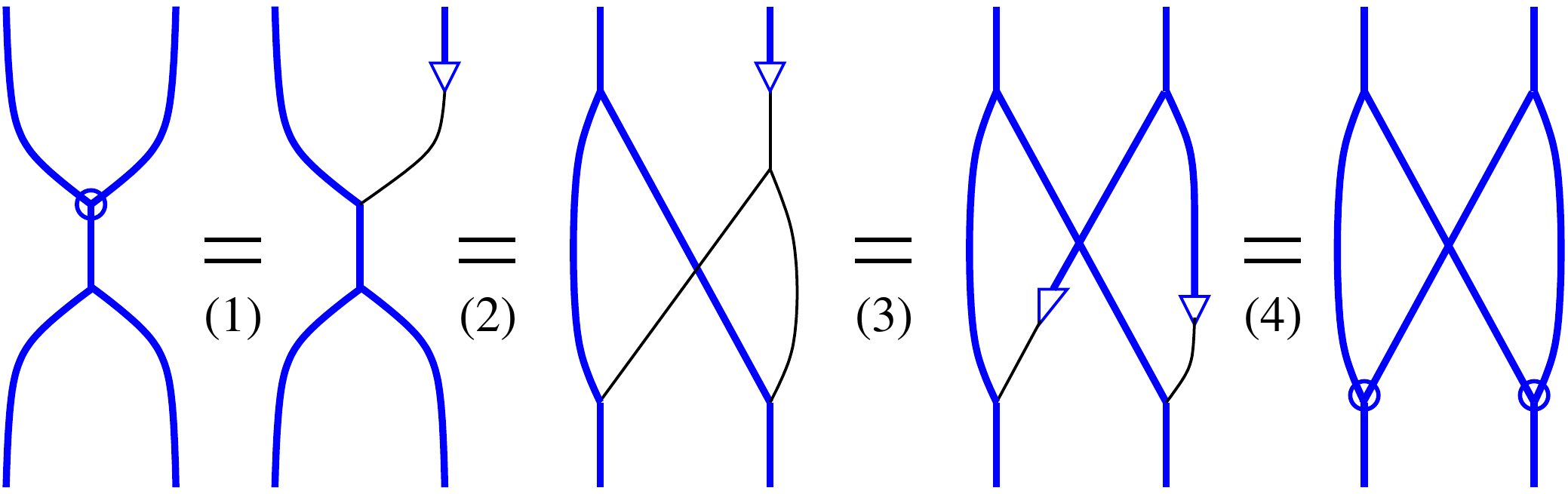}
\end{center}
\caption{Compatibility between SD operation and comultiplication}
\label{nucompati}
\end{figure}

\begin{definition}\cite{CCEScoalg} \label{def:R}
{\rm
Let $X$ be a categorical augmented rack $X$ with the augmentation map 
$\nu: X \rightarrow H$. 
Let $q: X\otimes X  \rightarrow X$ be the SD map associated with $\nu$. 
The  {\it  R-matrix associated with}  a categorical augmented rack $X$ with the augmentation map 
$\nu: X \rightarrow H$ is defined by $R(x \otimes y):= y^{(1)} \otimes q(x \otimes y^{(2)})$. This map is depicted in Figure~\ref{Rmatrix}.
}
\end{definition}

\begin{figure}[htb]
\begin{center}
\includegraphics[width=.5in]{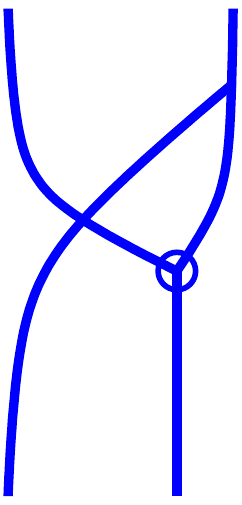}
\end{center}
\caption{The R-matrix associated with an SD map}
\label{Rmatrix}
\end{figure}

\begin{theorem}\label{thm:main}
Let $H$ be a Hopf algebra and $X$ be a cocommutative coalgebra on which $H$ acts, such that the action is compatible with the comultiplication of $X$.
Suppose $X$ is a categorical augmented rack with 
$\nu: X \rightarrow H$ that is a coalgebra morphism.

Then the R-matrix induced from the SD map associated with $\nu$,
as defined in Definition~\ref{def:R}  is indeed a YB operator.
\end{theorem}

\begin{proof}
In \cite{CCEScoalg}, Proposition 4.3,  it was proved that R-matrix induced from an SD map that is compatible with comultiplication
indeed satisfies the Yang-Baxter equation. For completeness the diagrams for the proof are included in Figure~\ref{YBE}. 
The portion of each equality applied is marked by dotted circles.
Equality (1) in the figure is a sequence of commutativity of the comultiplication and transpositions 
$(\Delta \otimes {\mathbb 1})\tau= ({\mathbb 1}\otimes \tau)(\tau \otimes {\mathbb 1})({\mathbb 1}\otimes \Delta)$ and the coassociativity and (2) is the commutativity of the SD map with transpositions.
Equations (3) and (4) are commutativities with transpositions of operations involved, 
and (5) is the SD condition.
Equality (6) is cocomutativity, (7) is the coassociativity and commutativity with transposition, and (8) 
is the compatibility between comultiplication and the SD map,
followed by again commutativity with transpositions in (9).
Finally cocommutativity in (10) and coassociativity in (11), together with commutativity with transpositions, 
complete the YBE.
The required compatibility is proved in Lemma~\ref{lem:SDcompati}.
The novel aspect of this proof comparing to  \cite{CCEScoalg} is the use of Lemma~\ref{lem:SDcompati}
through the structure of categorical augmentation rack.
\end{proof}

The invertibility of these YB operators will be discussed in Section~\ref{sec:BMC}.

\begin{figure}[htb]
\begin{center}
\includegraphics[width=5in]{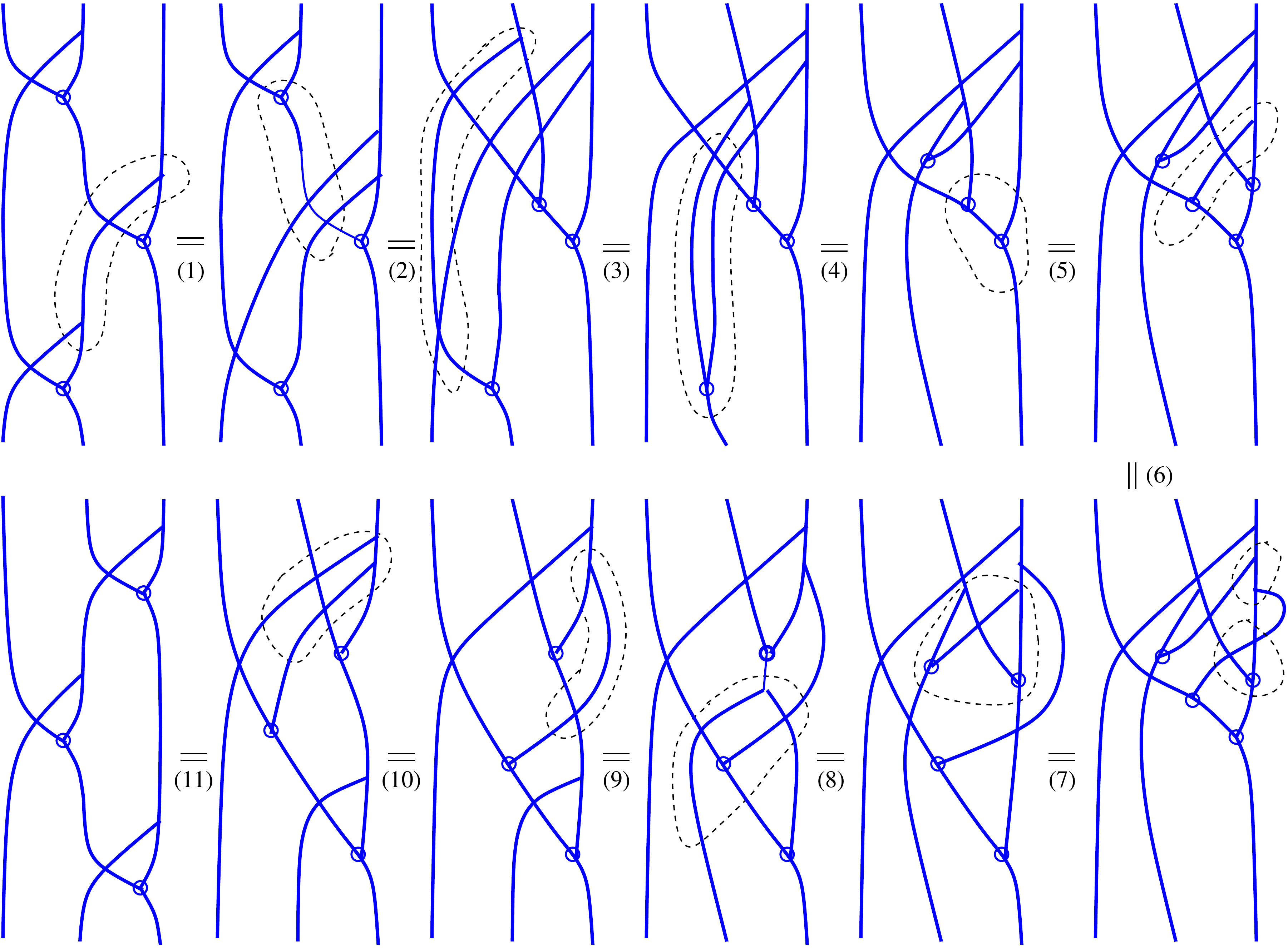}
\end{center}
\caption{YBE is satisfied by R-matrix associated with the SD operator}
\label{YBE}
\end{figure}

\section{From quantum heaps} \label{sec:heap}

In this section we provide  categorical augmented racks from quantum heaps and 
give a proof that the associated R-matrix is indeed a Yang-Baxter solution.
The construction is similar to Appendix B in \cite{ESZcascade}, where ternary augmented rack was discussed instead.

A {\it group heap} is a group $G$ with a ternary operation $T: G \times G \times G \rightarrow G$
defined by $T(x,y,z)=xy^{-1}z$, and is known to satisfy the ternary self-distributive (SD) law,
$$
		T(T(x,y,z),u,v)=T(T(x,u,v),T(y,u,v),T(z,u,v)), \quad x,y,z,u,v \in G. 
		$$ 
		Heaps have been studied recently in knot theory, due to the SD property,
		see for example \cites{EGM,ESZheap,ESZcascade}.
This operation can be regarded as $y^{-1}z$ acting on the right of $x$. In Hopf algebra, 
$y^{-1}z$ is reinterpreted by $S(y)z$, where $S$ is the antipode. This motivates the following definition,
and also the motivation of naming the corresponding structure a {\it quantum heap}. 
Quantum heaps have also been used in knot theory~\cites{SZframedlinks,SZbrfrob}.

\begin{lemma}\label{lem:heap}
Let $H$ be a cocommutative Hopf algebra.
Set $X := H \otimes H$ and define $\nu: X \rightarrow H$ by $\nu(x \otimes y) := S(x) y
(=\mu(S(x) \otimes  y)$. 
Also define the action of $H$ on $X$ by 
$$(x \otimes y) \cdot g:= (x \otimes y) \Delta_H(g)
(=x g^{(1)} \otimes y g^{(2)}). $$
Then $\nu$ defines a categorical augmented rack structure on  $X$.
\end{lemma}

\begin{proof}
The map $\nu$ defined is represented by Figure~\ref{heapnu} (A). 
Define the comultiplication on $X$ by 
$\Delta_X ( x \otimes y) = (x^{(1)} \otimes y^{(1)} ) \otimes (x^{(2)} \otimes y^{(2)} )$,
where $\Delta_H (x)=x^{(1)} \otimes x^{(2)} $ and $\Delta_H (y)=y^{(1)} \otimes y^{(2)} $,
as depicted in Figure~\ref{heapnu} (B). 
The action is depicted in (C). 

Then a proof for the commutativity of $\nu$ and $\Delta_H$ is indicated in Figure~\ref{heapcomult}. 
Specifically, (1) is the definition, (2) is the compatibility between multiplication and comultiplication of a Hopf algebra, 
(3) is the equality between $S$ and $\Delta$, (4) is the cocommutativity, and (5) is the definition.

A proof for the augmentation map condition is indicated in Figure~\ref{heapcataug}. 
Specifically, (1) is the definition, (2) is the relation between $\mu$ and $S$, (3) and (4)  are associativity, 
and (5) is the definition.
The last figure is similar to the figure in \cite{ESZcascade} Appendix B, where the diagrams were used for ternary augmentation instead.
\end{proof}

\begin{figure}[htb]
\begin{center}
\includegraphics[width=3.2in]{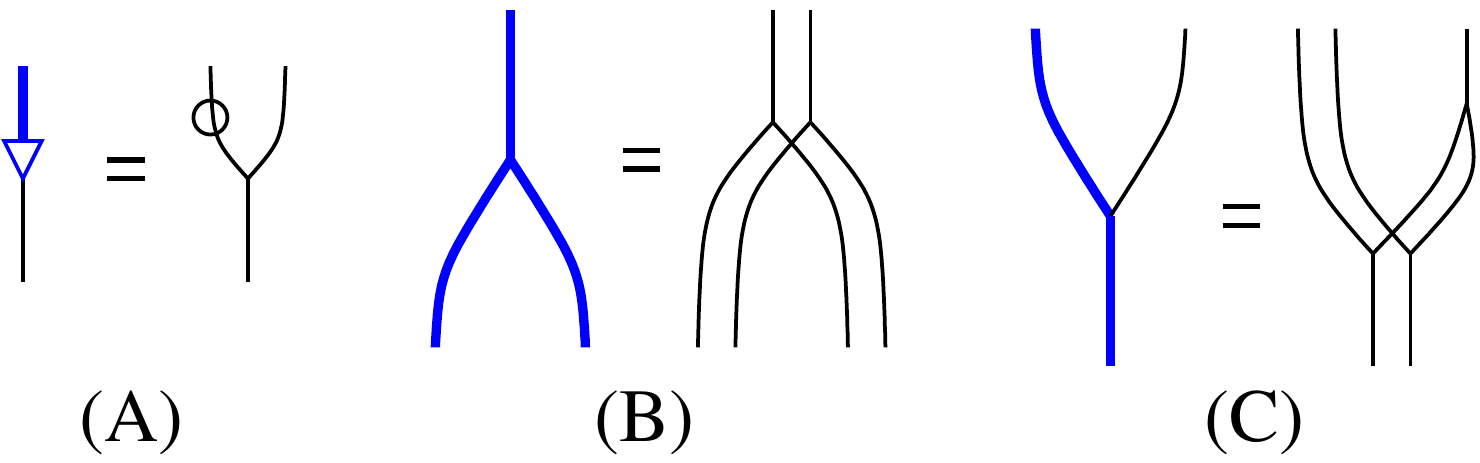}
\end{center}
\caption{Comultiplication and  $\nu$ for quantum heaps}
\label{heapnu}
\end{figure}

\begin{proposition}
Let $H$ be a Hopf algebra.
Set $X := H \otimes H$ and define  the categorical augmentation map $\nu$
and the action of $H$ on $X$ as in 
Lemma~\ref{lem:heap},
$\nu: X \rightarrow H$ by $\nu(x \otimes y) := S(x) y
=m(S(x) \otimes  y)$,
and 
y $(x \otimes y) \cdot g:= (x \otimes y) \Delta_H(g)
=(x g^{(1)}) \otimes (y g^{(2)})$.
Then the R-matrix associated to $\nu$ is indeed a solution to the Yang-baxter equation.
\end{proposition}

\begin{proof}
This follows from Theorem~\ref{thm:main} and Lemma~\ref{lem:heap}.
\end{proof}

\begin{figure}[htb]
\begin{center}
\includegraphics[width=4.5in]{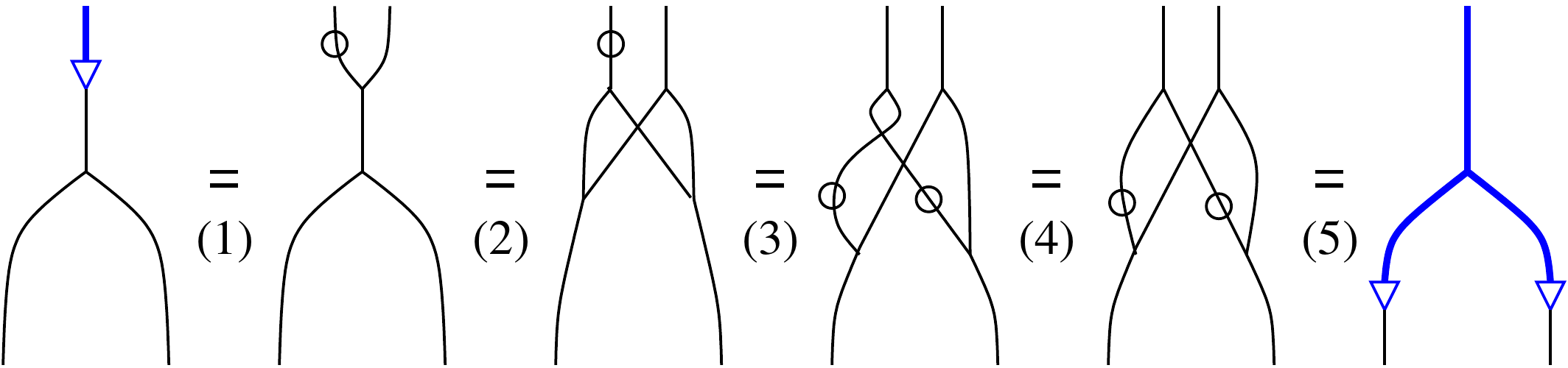}
\end{center}
\caption{Commutativity of $\nu$ with comultiplication for quantum heaps}
\label{heapcomult}
\end{figure}

\begin{figure}[htb]
\begin{center}
\includegraphics[width=4.5in]{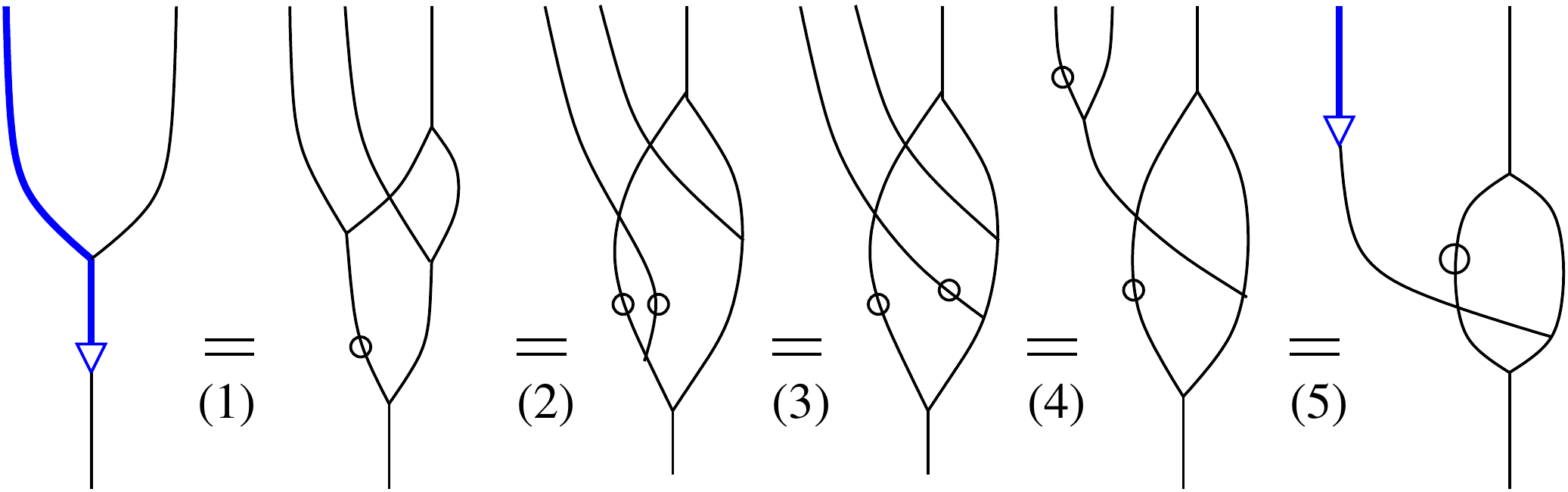}
\end{center}
\caption{Augmentation map for quantum heaps}
\label{heapcataug}
\end{figure}

\section{From adjoint of Hopf algebras}\label{sec:ad}

In this section we show that the doubling of adjoint map of a Hopf algebra gives rise to 
a YB operator using the categorical augmented rack structure.

\begin{definition}[e.g.~\cite{CCESadj}] 
{\rm
Let $H$ be a Hopf algebra and let $y \in H$.
Define the {\it adjoint map} $ad_y: H \rightarrow H$ by $ad_y(x):= S(y^{(1)}) x y^{(2)}$.
}
\end{definition}

A diagram representing the adjoint map is the same as Figure~\ref{aug} (B) except there is no 
blue line and triangle at the top left corner.

\begin{figure}[htb]
\begin{center}
\includegraphics[width=3.2in]{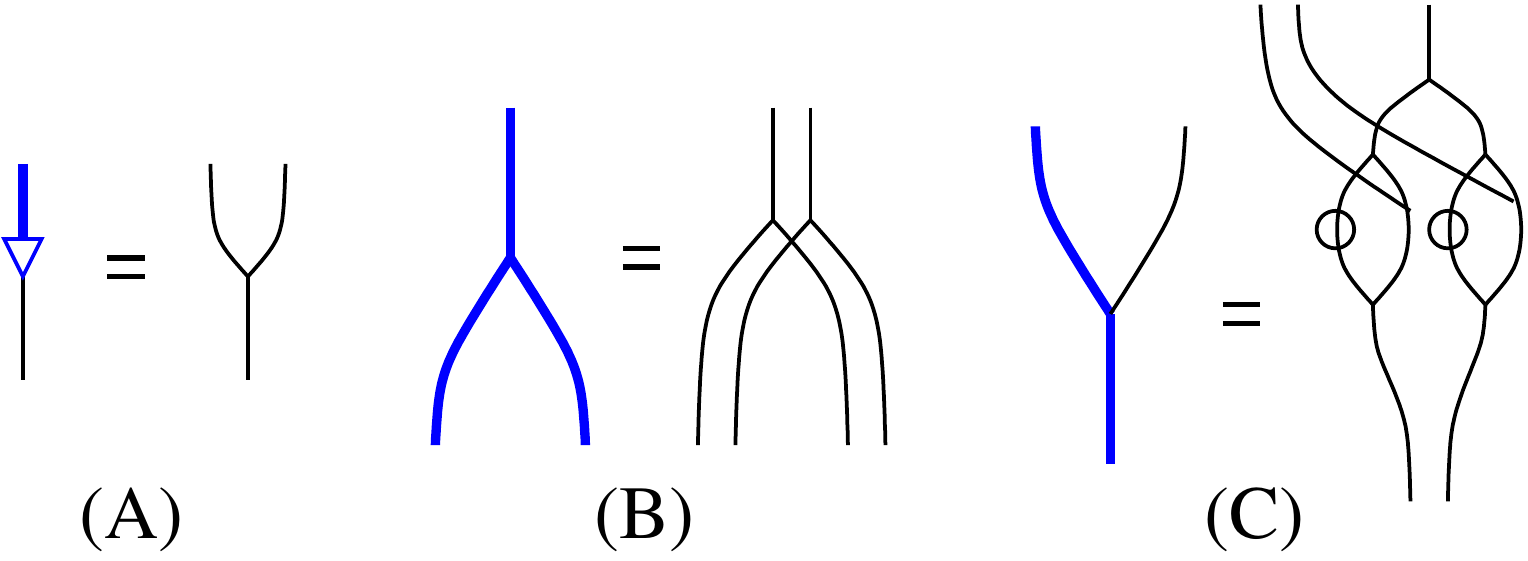}
\end{center}
\caption{Comultiplication and  $\nu$ for adjoint maps}
\label{ad}
\end{figure}

\begin{figure}[htb]
\begin{center}
\includegraphics[width=4in]{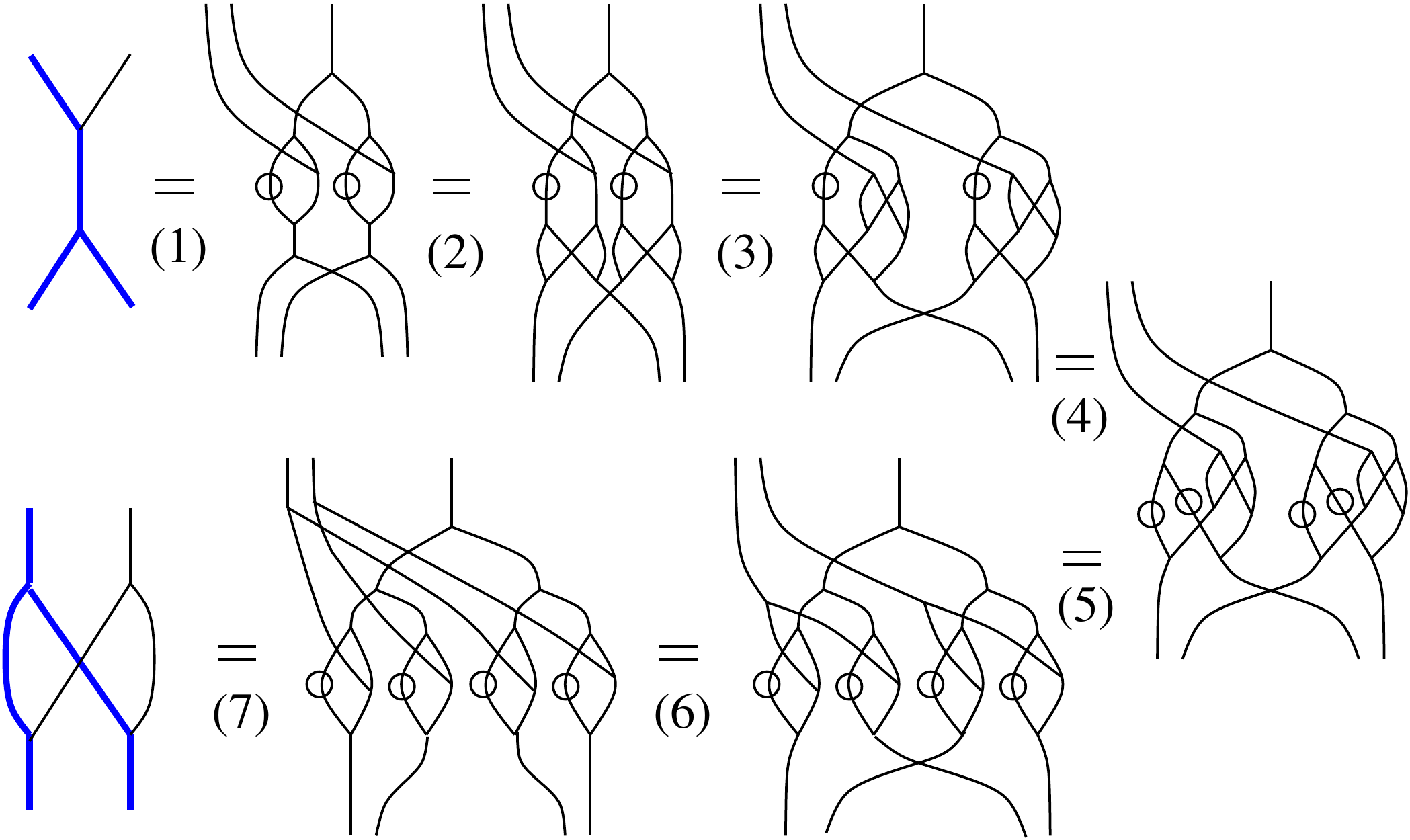}
\end{center}
\caption{Compatibility of the action and $\Delta_X$ for the adjoint}
\label{adaction}
\end{figure}

\begin{lemma}\label{lem:double}
Let $H$ be a cocommutative Hopf algebra.
Set $X := H \otimes H$ and define $\nu: X \rightarrow H$ by $\nu(x \otimes y) := x y
(=\mu(x \otimes  y) )$. 
Also define the action of $H$ on $X$ by $(x \otimes y) \cdot g:= ad_{g^{(1)} }(x) \otimes ad_{g^{(2)}} (y)$.
Then $\nu$ defines a categorical augmented rack structure on  $X$.
\end{lemma}

\begin{proof}
The map $\nu$ defined is represented by Figure~\ref{ad} (A). 
Define the comultiplication on $X$ by 
$\Delta_X ( x \otimes y) = (x^{(1)} \otimes y^{(1)} ) \otimes (x^{(2)} \otimes y^{(2)} )$,
where $\Delta_H (x)=x^{(1)} \otimes x^{(2)} $ and $\Delta_H (y)=y^{(1)} \otimes y^{(2)} $,
as depicted in Figure~\ref{ad} (B). 
The action is depicted in (C). 
A proof for the compatibility between the action  and $\Delta_X$ is depicted in Figure~\ref{adaction}.
Specifically, (1) is the definition, (2) and (3) are the compatibility between multiplication and comultiplication, 
(4) is the relation between $S$ and $\Delta$ together with cocommutativity, 
 (5) and (6)  are the commutativity with  $\tau$, and (7) is the definition.

The commutativity of $\nu$ and $\Delta_H$ follows from the compatibility between multiplication and comultiplication 
as  depicted in Figure~\ref{adcomult}. 
The augmentation map condition is depicted in Figure~\ref{double}. 
Specifically, (1) is the definition, (2) is (co)associativity, (3) is (co)associativity together with the antipode condition
(Figure~\ref{hopfaxiom} (D)), (4) is the associativity, and (5) is the commutativity with $\tau$ together with the definition.
The last figure is similar to the figure in \cite{ESZcascade} Appendix B, where the diagrams were used for ternary augmentation instead.
\end{proof}

\begin{figure}[htb]
\begin{center}
\includegraphics[width=2.5in]{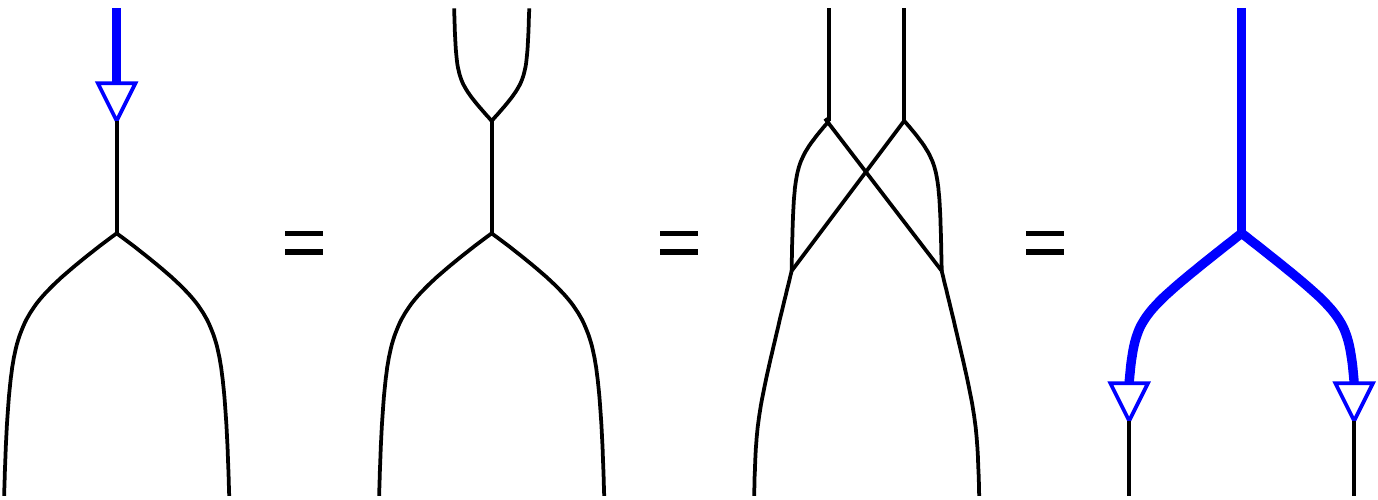}
\end{center}
\caption{Commutativity of $\nu$ with comultiplication for adjoint maps}
\label{adcomult}
\end{figure}

\begin{proposition}
Let $H$ be a Hopf algebra.
Set $X := H \otimes H$ and define  the categorical augmentation map $\nu$
and the action of $H$ on $X$ as in 
Lemma~\ref{lem:double},
by $\nu(x \otimes y) := x y
(=\mu(x \otimes  y ) )$, and
 $(x \otimes y) \cdot g:= ad_{g^{(1)} }(x) \otimes ad_{g^{(2)}} (y)$.
Then the R-matrix associated to $\nu$ is indeed a solution to the Yang-baxter equation.
\end{proposition}

\begin{proof}
This follows from Theorem~\ref{thm:main} and Lemma~\ref{lem:double}.
\end{proof}

\begin{figure}[htb]
\begin{center}
\includegraphics[width=4.5in]{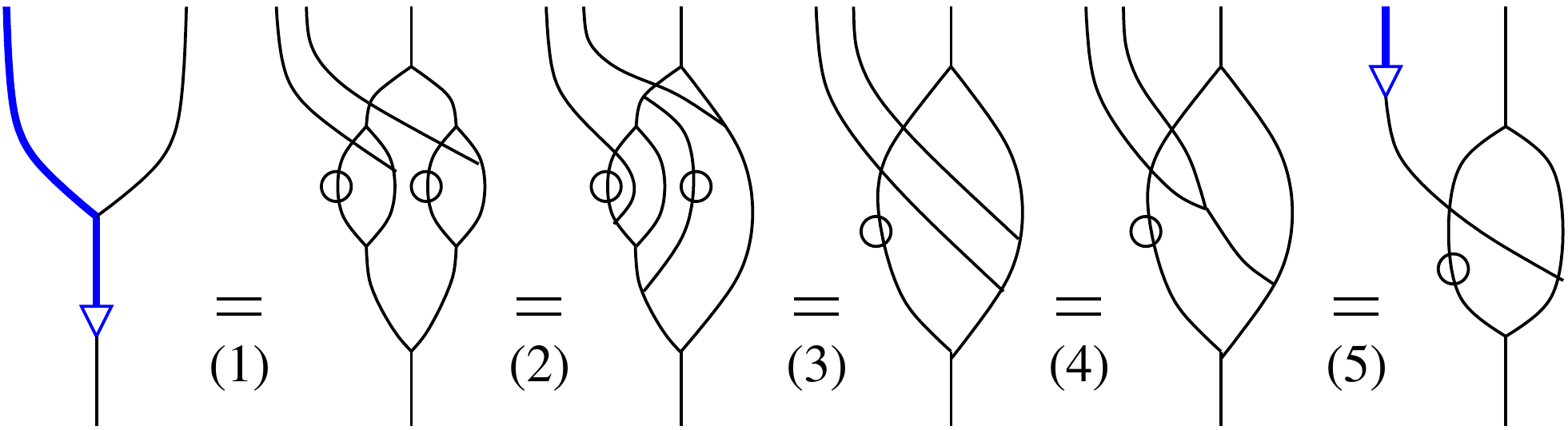}
\end{center}
\caption{Augmentation map for adjoint maps}
\label{double}
\end{figure}

\section{Doubling construction via multiplication}\label{sec:doubling}

In this section we provide a method of inductive construction of  categorical augmented racks
by doubling and multiplication.

\begin{figure}[htb]
\begin{center}
\includegraphics[width=3in]{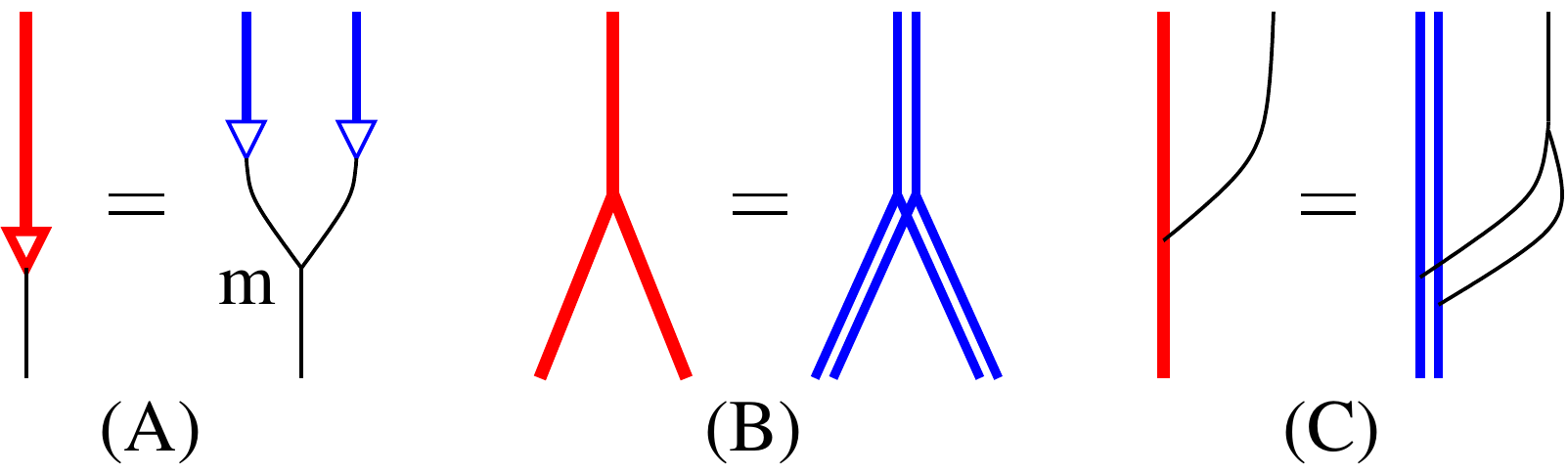}
\end{center}
\caption{Doubling defined inductively}
\label{inductive}
\end{figure}

\begin{theorem}\label{thm:ind}
As in Theorem~\ref{thm:main}, 
let $H$ be a Hopf algebra and $X$ be a coalgebra on which $H$ acts, such that the action is compatible with the comultiplication of $X$.
Suppose $X$ is a categorical augmented rack with 
$\nu: X \rightarrow H$ that is a coalgebra morphism.

Let $Y:=X \times X$ be equipped with the coalgebra structure inherited from $X$,
$$\Delta_Y := ({\mathbf 1} \otimes \tau \otimes {\mathbf 1}) ( \Delta_X \otimes \Delta_X) : Y \rightarrow Y \otimes Y.$$
Equip $Y$ with the right action via $X$ through $\Delta_X$ as well, $y \cdot h:=y \cdot \Delta_H(h)$.

Define $\tilde{\nu} : Y \rightarrow H$ by $\tilde{\nu} := \mu (\nu \otimes \nu)$, where $\mu$ is the multiplication on $H$. Then $Y$ is a categorical augmented rack with $\tilde{\nu}$.
\end{theorem}

\begin{proof}
The definition of $\tilde{\nu}$ is depicted in Figure~\ref{inductive} (A), 
The definition of the comultiplication on $Y$ and  the action of $H$ on $Y$ are depicted in (B) and (C), respectively.
The augmentation condition for $\tilde{\nu}$ is depicted in Figure~\ref{indaug}.
Specifically, (1) and (2) are definitions, (3) and (4) follow from (co)associaivity and the antipode condition, 
(5) is associativity and (6) is the definition.
The result follows.
\end{proof}

\begin{figure}[htb]
\begin{center}
\includegraphics[width=5.3in]{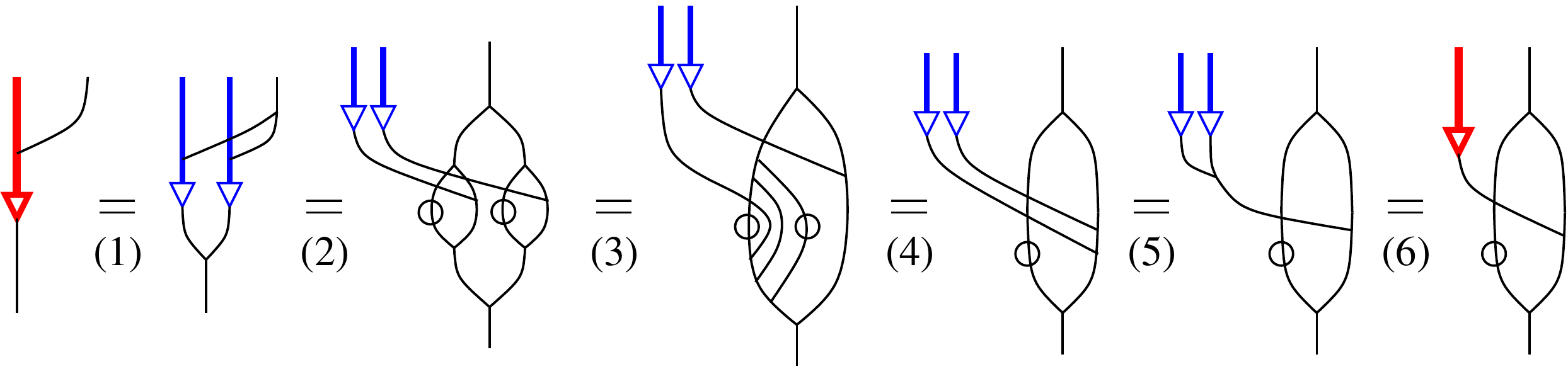}
\end{center}
\caption{Augmentation condition for the inductive construction}
\label{indaug}
\end{figure}

The following theorem constructs an infinite family of YBE solutions inductively using 
Theorems~\ref{thm:main} and \ref{thm:ind}. 

\begin{theorem} \label{thm:indR}
Let $H$ be a cocommutative Hopf algebra and $X$ be a coalgebra on which $H$ acts, such that the action is compatible with the comultiplication of $X$.
Suppose $X$ is a categorical augmented rack with 
$\nu: X \rightarrow H$ that is a coalgebra morphism.
Let $(Y, \tilde{\nu})$ be a categorical augmented rack constructed from $X$ and $H$ as in Theorem~\ref{thm:ind}.
Then the associated R-matrix satisfies YBE.
\end{theorem}

\begin{proof}
We apply Theorem~\ref{thm:main} to $(Y, \tilde{\nu})$. 
The compatibility between the action and the comultiplication is shown in Figure~\ref{indcompati},
which consists of definitions and the middle equality
$$
( \Delta_X \otimes \Delta_X) ( (x \cdot g^{(1)} )  \otimes ( y \cdot g^{(2)}) ) 
=
(\Delta_X(x) \cdot g^{(1)})  \otimes (\Delta_X(y) \cdot g^{(2)}) .
$$
The left  hand side of the equality is computed as 
\begin{eqnarray*}
\lefteqn{ ( \Delta_X \otimes \Delta_X) ( (x \cdot g^{(1)} )  \otimes ( y \cdot g^{(2)}) ) }\\
&=& 
[ ( x^{(1)} \cdot g^{(11)} ) \otimes (x^{(2)} \cdot g^{(12)} ) ]  \otimes [ ( y^{(1)} \cdot g^{(21)} ) \otimes 
(y^{(2)} \cdot g^{(22)} ) ] 
\end{eqnarray*}
and, using the compatibility between the action and comultiplication,  the right hand side is 
\begin{eqnarray*}
\lefteqn{  (\Delta_X(x) \cdot g^{(1)})  \otimes (\Delta_X(y) \cdot g^{(2)}) }\\
&=& [ ( x^{(1)} \cdot g^{(11)} ) \otimes (y^{(1)} \cdot g^{(12)} ) ]  \otimes [ ( x^{(2)} \cdot g^{(21)} ) \otimes (y^{(2)} \cdot g^{(22)} ) ].
\end{eqnarray*}
Then the equality holds by coassociativity and  cocommutativity.

The property of $\tilde{\nu}$ being a coalgebra morphism is depicted in Figure~\ref{indcomulti}, that follows from
the compatibility conditions of multiplication and comultiplication together with the commutativity between $\nu$ and $\Delta$ as depicted.
Thus Theorem~\ref{thm:main} applies.
\end{proof}

In particular, constructions given in Sections~\ref{sec:heap} and \ref{sec:ad} provide infinite families of YB operators through Theorem~\ref{thm:indR} inductively.

\begin{figure}[htb]
\begin{center}
\includegraphics[width=2.5in]{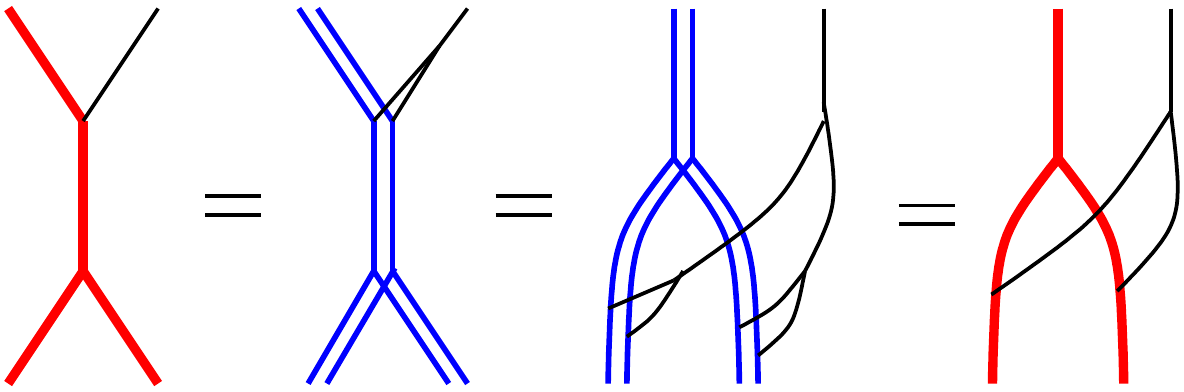}
\end{center}
\caption{Compatibility between action and comultiplication for the inductive construction}
\label{indcompati}
\end{figure}

\begin{figure}[htb]
\begin{center}
\includegraphics[width=3in]{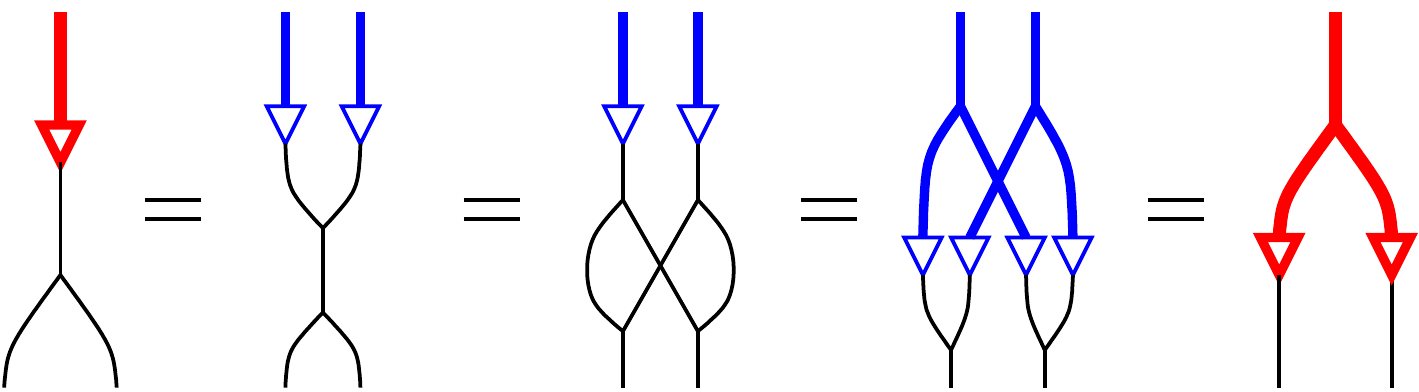}
\end{center}
\caption{Compatibility between augmentation map and comultiplication}
\label{indcomulti}
\end{figure}

\section{Construction of a braided monoidal category}\label{sec:BMC}

In this section we generalize the doubling construction of Theorem~\ref{thm:indR}
to construct a braided monoidal category generated by one object.
 
 A {\it strict monoidal category} of a category ${\cal C}$ is equipped with a bifunctor 
 $\otimes : {\cal C} \times {\cal C} \rightarrow {\cal C}$ that is associative (the associator is the identity)
 and the unitors $X \otimes I \rightarrow X$, $I \otimes X \rightarrow X$ are identities.
 A {\it braided strict monoidal category} is a strict monoidal category equipped with 
 {\it braiding} $R_{A,B} : A \otimes B \rightarrow B \otimes A$, that are natural isomorphisms, for all objects $A,B$ satisfying 
 \begin{eqnarray}
 R_{A, B \otimes C} = ({\mathbb 1} \otimes R_{A, C} ) (R_{A,B} \otimes {\mathbb 1} ) &:& 
 A \otimes B \otimes C \rightarrow B \otimes C \otimes A  , \label{BMC1} \\
 R_{A \otimes B , C} = (R_{A,C} \otimes {\mathbb 1} )  ({\mathbb 1} \otimes R_{B, C} ) &:& 
  A \otimes B \otimes C \rightarrow C \otimes A \otimes B . \label{BMC2}
  \end{eqnarray}
  The braiding satisfies the braid relation
  $$ 
( R_{B, C} \otimes {\mathbb 1}_A ) ({\mathbb 1}_B \otimes R_{A, C} ) ( R_{A,B} \otimes {\mathbb 1}_C )
=
 ({\mathbb 1}_C \otimes R_{A,B} ) ( R_{A, C} \otimes {\mathbb 1}_B )  ({\mathbb 1}_A \otimes R_{B, C} ) .
 $$
Two braided monoidal categories  are {\it braided equivalent} if 
there is a functor  of the two monoidal categories that commutes with braidings.

Let $X$ be a categorical augmentation rack over a Hopf algebra $H$ with $\nu: X \rightarrow H$ 
 as in Theorem~\ref{thm:indR}.

   Let ${\mathcal B}(X,H,\nu)$ be the category defined as follows.
 The objects consists of right $H$-modules with comultiplications compatible with the $H$-action, $X^{\otimes n}$ for non-negative integers, where $X^0= {\mathbb k}$ is the coefficient ring, so that ${\mathcal B}(X,H,\nu)$ is generated by a single object $X$.
  The objects consists of $X^{\otimes n}$ for non-negative integers, where $X^0= {\mathbb k}$ is the coefficient ring.
  Thus objects are right $H$-modules with comultiplications compatible with the $H$-action, 
  and ${\mathcal B}(X,H,\nu)$ is generated by a single object $X$.
 
 For each $n$, comultiplication 
 $\Delta_n: X^{\otimes n} \rightarrow X^{\otimes n} \otimes X^{\otimes n}$
 is defined by
 $$\Delta_n (x_1 \otimes  \cdots \otimes  x_n)=(x_1^{(1)} \otimes  \cdots \otimes x_n^{(1)} ) \otimes (x_1^{(2)} \otimes  \cdots \otimes x_n^{(2)} ) . $$
 This defines a coassociative comultiplication.
For each $n$, the map $\nu_n: X^{\otimes n} \rightarrow H$ is defined by 
$\nu_n:=M \circ \nu^{\otimes n}$, 
where $M:=\mu (\mu \otimes {\mathbb 1}) \otimes (\mu \otimes {\mathbb 1}^{\otimes 2 } )
\cdots (\mu \otimes {\mathbb 1}^{\otimes (n-1) }) $. 
The braiding
$$R_{m,n}: X^{\otimes m} \otimes X^{\otimes n} \rightarrow X^{\otimes n} \otimes X^{\otimes m} $$
 is defined as in Definition~\ref{def:R} by
 $R_{m,n} (x \otimes y) = y^{(1)} \otimes q(x \otimes y^{(2)} ) $, where 
 for $y \in X^{\otimes n}$, $\Delta_n(y) = y^{(1)} \otimes y^{(2)} $ is as defined above,
 and $q$ is the SD map induced from the map $\nu_n$, 
 $q(x \otimes y) = x \cdot \nu_n(y)$ as in Definition~\ref{def:SD}, using the 
 action $(x_1 \otimes \cdots \otimes x_n) \cdot g = (x_1 \cdot g^{(1)} \otimes \cdots \otimes x_n \cdot g^{(n)} ) $.
  A diagrammatic representation of the definition of $R_{m,n}$ is depicted in Figure~\ref{fig:RmnDef}.
  
 \begin{figure}[htb]
 	\begin{center}
 		\includegraphics[width=1.2in]{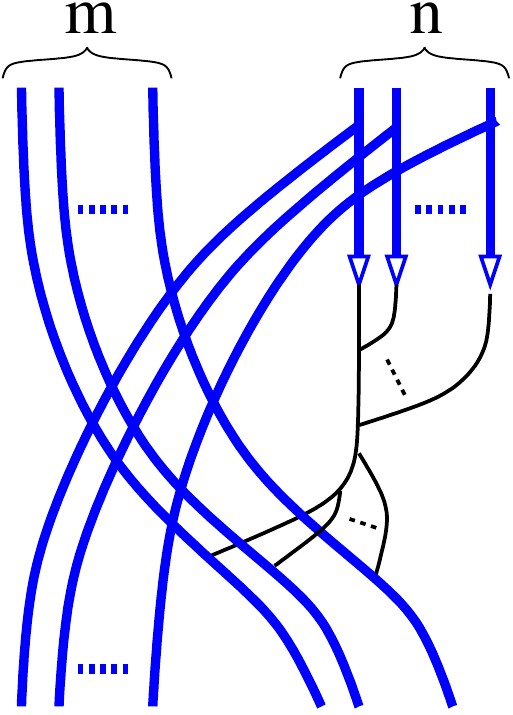}
 	\end{center}
 	\caption{Diagrammatic depiction of definition of $R_{m,n}$.}
 	\label{fig:RmnDef}
 \end{figure}

 We define the morphisms 
 ${\rm Hom}(X^{\otimes k}, X^{\otimes \ell})$  
 to consist of  the zero map  $X^{\otimes k} \rightarrow X^{\otimes \ell}$ for $k \neq \ell$,
  and for $k = \ell =n$ we set ${\rm Hom}(X^{\otimes n}, X^{\otimes n})$ to be the monoid 
 under composition generated by the maps $X^{\otimes n} \rightarrow X^{\otimes n}$ of type 
 $$\mathbb 1^{\otimes k_1}\otimes R_{s, t} \otimes \mathbb 1^{\otimes k_2}\quad {\rm and} \quad \mathbb 1^{\otimes k_1}\otimes R_{s, t}^{-1} \otimes \mathbb 1^{\otimes k_2}, $$
  such that $k_1 + s + t + k_2 = n$. Here we are implicitly assuming that taking a power zero composition of such morphisms gives the identity, so that ${\rm Hom}(X^{\otimes n}, X^{\otimes n})$ contains the identity morphism, as it should.

  \begin{theorem}
 	The category $\mathcal B(X,H,\nu)$ constructed above is indeed a braided monoidal category. Moreover, two such categories $\mathcal B(X_1 ,H, \nu_1)$ and $\mathcal B(X_2, H, \nu_2)$ are braided equivalent if $(X_1, H, \nu_1)$ and $(X_2, H, \nu_2)$ are equivalent as categorical augmented racks. When the actions of $H$ on $X_1$ and $X_2$ are faithful,
	then the converse to the preceding statement also holds. 
 	
 \end{theorem}

 \begin{proof}

Before proving the properties that define a braided monoidal category, there are some well posedeness assumptions made in the definition of $\mathcal B(X,H,\nu)$ above that need to be verified. Namely, we need to check that the braidings 
$R_{m,n}$ 
are invertible, and that the tensor product of two morphisms in 
${\rm Hom}(X^{\otimes k},X^{\otimes \ell})$ 
is a morphism as well. To show the first fact, i.e. the invertibility of the braidings, let us define 
$R^{-1}_{n,m} :   X^{\otimes n} \otimes X^{ \otimes m} 
\rightarrow  X^{ \otimes m} \otimes X^{\otimes n} $. 
 through the assignment on simple tensors given by
$$
		x_1\otimes \cdots \otimes x_n \Motimes y_1 \otimes \cdots \otimes y_m \mapsto q^{-1}(y_1\otimes \cdots \otimes y_m \Motimes x_1^{(2)}\otimes \cdots \otimes x_n^{(2)}) \Motimes x_1^{(1)}\otimes \cdots \otimes x_n^{(1)},
$$ 
where
$$
	q^{-1}(x_1\otimes \cdots \otimes x_n \Motimes y_1\otimes \cdots \otimes y_m) := (x_1\otimes \cdots \otimes x_n)\cdot S(\nu_m(y_1\otimes \cdots \otimes y_m)).
$$
We now show that $R^{-1}_{n,m}$ is an inverse for $R_{m,n}$. On simple tensors we have
\begin{eqnarray*}
\lefteqn{R^{-1}_{n,m}R_{m,n}(x_1\otimes \cdots \otimes x_m \Motimes y_1\otimes \cdots \otimes y_n)}\\
 &=& R^{-1}_{n,m}(y_1^{(1)}\otimes \cdots \otimes y_n^{(1)}\Motimes x_1\otimes \cdots \otimes x_m)\cdot \nu_m(y_1^{(2)}\otimes \cdots \otimes y_n^{(2)}))\\
 &=& R^{-1}_{n,m}(y_1^{(1)}\otimes \cdots \otimes y_n^{(1)}\Motimes x_1\otimes \cdots \otimes x_m \cdot \nu(y_1^{(2)})\cdots \nu(y_n^{(2)}))\\
 &=& [(x_1\otimes \cdots \otimes x_m
 \cdot \nu(y_1^{(1)})\cdots \nu(y_n^{(1)})]\cdot S(\nu_n(y_1^{(2)}\otimes \cdots \otimes y_n^{(2)})) \Motimes y_1^{(3)}\otimes \cdots \otimes y_n^{(3)}\\
 &=& [(x_1\otimes \cdots \otimes x_m)\cdot \nu(y_1^{(1)})\cdots \nu(y_n^{(1)})]\cdot S(\nu(y_1^{(2)})\cdots \nu(y_n^{(2)}))\Motimes y_1^{(3)}\otimes \cdots \otimes y_n^{(3)}\\
 &=& (x_1\otimes \cdots \otimes x_m)\cdot [\nu(y_1^{(1)})\cdots \nu(y_n^{(1)})\cdot S(\nu(y_1^{(2)})\cdots \nu(y_n^{(2)}))] \Motimes y_1^{(3)}\otimes \cdots \otimes y_n^{(3)}\\
 &=& (x_1\otimes \cdots \otimes x_m)\cdot [\nu(y_1^{(1)})\cdots \nu(y_n^{(1)})\cdot S(\nu(y_n^{(2)}))\cdots S(\nu(y_1^{(2)}))]\Motimes y_1^{(3)}\otimes \cdots \otimes y_n^{(3)}\\
 &=& x_1\otimes \cdots \otimes x_m \Motimes [ \epsilon(y_1^{(1)})\cdots \epsilon(y_n^{(1)}) ] y_1^{(2)}\otimes \cdots \otimes y_n^{(2)}\\
 &=& x_1\otimes \cdots \otimes x_m \Motimes y_1\otimes \cdots \otimes y_n.
\end{eqnarray*}

A similar approach also shows that 
$R_{n,m}R^{-1}_{n,m} = {\mathbb 1}_{X^{\otimes (m+n)} } $. 
Therefore, the braidings $R_{m,n}$ that we have defined are invertible. 

We now show that taking tensor products is well defined within the category. For objects this is obvious. Let us consider two morphisms. First, if we take the tensor product of morphisms where at least one of the tensorand is an element of 
${\rm Hom}(X^{\otimes k},X^{\otimes \ell})$ for $k \neq \ell$, we obtain the zero morphism, which is always a morphism in the category. 
Consider a tensor product of type $(f_u\circ \cdots \circ  f_1) \otimes (g_v\circ \cdots \circ g_1)$, where all $f_i$ and $g_j$ are morphisms of type 
$\mathbb 1^{\otimes k_1}\otimes R^{\pm}_{s,t} \otimes \mathbb 1^{\otimes k_2}$, 
where $f_i : X^{\otimes m} \rightarrow X^{\otimes m}$ and $g_j : X^{\otimes n} \rightarrow X^{\otimes n}$. We can write this tensor product as
\begin{eqnarray*}
		(f_u\circ \cdots f_1) \otimes (g_v\circ \cdots \circ g_1) &=& (g_v\otimes \mathbb 1^{\otimes m})\circ \cdots \circ (g_1\otimes \mathbb 1^{\otimes m}) \circ (\mathbb 1^{\otimes n}\otimes f_u) \circ \cdots \circ (\mathbb 1^{\otimes n}\otimes f_u),
\end{eqnarray*}
which is again a composition of morphisms of the same type $\mathbb 1^{\otimes k_1}\otimes R^{\pm}_{s, t} \otimes \mathbb 1^{\otimes k_2}$. This shows that taking tensor products in the category is well defined. 

Next, we show that the braiding defined above satisfies 
the axioms of the braided monoidal category, Equation (\ref{BMC1}) and (\ref{BMC2}).
In Equation (\ref{BMC1}), let $A=X^{\otimes \ell  }$, $B=X^{ \otimes m}$ and $C=X^{\otimes n}$.
For $x \in A$, $y \in B$ and $z \in C$, the left and right hand sides of (\ref{BMC1}) 
are computed, respectively,
\begin{eqnarray*}
R_{A, B \otimes C}(x \otimes (y \otimes z)) 
&=&  
( y \otimes z)^{(1)} \otimes x \cdot \nu_{m+n} ( (y \otimes z)^{(2)} ) \\
&=& 
( y^{(1)} \otimes z^{(1)} ) \otimes ( x \cdot \nu_m(y^{(2)})  \nu_n (z^{(2)} ) ), \\
({\mathbb 1} \otimes R_{A,  C} )( R_{A,B} \otimes {\mathbb 1} ) ( ( x \otimes y) \otimes z)
&=&
({\mathbb 1} \otimes R_{A,  C} ) ( y^{(1)} \otimes ( x \cdot \nu_m( y^{(2)} ) \otimes z) )\\
&=& 
y^{(1)} \otimes z^{(1)} \otimes ( x \cdot \nu_m (y^{(2)})  \nu_n (z^{(2)} ) ),
\end{eqnarray*}
and Equation (\ref{BMC1}) follows. Equation  (\ref{BMC2}) is similar.

To prove naturality of braidings, we
observe that if we show that the any $R_{m,n}$ can be decomposed in appropriate product of braidings of type $R_{1,1}$ (tensored by an appropriate number of identity morphisms), then the result would follow by decomposing the $R_{m,n}$ and then using the braid relation for $R_{1,1}$ (proved in Theorem~\ref{thm:main}) several times. We show such a decomposition for $R_{2,2}$ for notational simplicity, since the generalization to arbitrary $m,n$ is substantially the same process but with more cumbersome notation. We have
\begin{eqnarray*}
		R_{2,2}(x_1\otimes x_2\Motimes y_1\otimes y_2) &=& y_1^{(1)} \otimes y_2^{(1)} \Motimes (x_1\otimes x_2)\cdot (\nu(y_1^{(2)})\nu(y_2^{(2)}))\\
		&=& y_1^{(1)} \otimes y_2^{(1)} \Motimes x_1\cdot (\nu(y_1^{(2)})\nu(y_2^{(2)})) \otimes x_2\cdot (\nu(y_1^{(3)})\nu(y_2^{(3)}))
\end{eqnarray*}
where we have used the compatibility of $\nu$ and comultiplication $\Delta$. Also, we have
\begin{eqnarray*}
		\lefteqn{(\mathbb 1\otimes R_{1,1}\otimes \mathbb 1)(\mathbb 1^{\otimes 2} \otimes R_{1,1})(R_{1,1}\otimes \mathbb 1^{\otimes 2})(\mathbb 1\otimes R_{1,1}\otimes \mathbb 1)(x_1\otimes x_2\Motimes y_1\otimes y_2)}\\
		 &=& (\mathbb 1\otimes R_{1,1}\otimes \mathbb 1)(\mathbb 1^{\otimes 2} \otimes R_{1,1})(R_{1,1}\otimes \mathbb 1^{\otimes 2})(x_1\otimes y_1^{(1)} \Motimes x_2\cdot \nu(y_1^{(2)}) \otimes y_2)\\
		 &=& (\mathbb 1\otimes R_{1,1}\otimes \mathbb 1)(\mathbb 1^{\otimes 2} \otimes R_{1,1})(y_1^{(1)}\otimes x_1\cdot \nu(y_1^{(2)})\Motimes x_2\cdot \nu(y_1^{(3)}) \otimes y_2)\\
		 &=& (\mathbb 1\otimes R_{1,1}\otimes \mathbb 1)(y_1^{(1)}\otimes x_1\cdot \nu(y_1^{(2)})\Motimes y_2^{(1)}\otimes x_2\cdot \nu(y_1^{(3)})\nu(y_2^{(2)}))\\
		 &=& y_1^{(1)} \otimes y_2^{(1)} \Motimes x_1\cdot \nu(y_1^{(2)})\nu(y_2^{(2)}) \otimes x_2\cdot \nu(y_1^{(3)})\nu(y_2^{(3)}).
\end{eqnarray*}
Therefore, we have shown that $R_{2,2}$ is a composition of $R_{1,1}$ 
 tensored with identities.
 Similarly  it is shown that 
  any $R_{m,n}$ can be decomposed as a composition of $R_{1,1}$ (tensored with  identities) by taking the permutation of $(m+n)$ elements that exchanges the first $m$ and the last $n$, decomposing it into product of transpositions, and then identifying transpositions and morphisms of type $R_{1,1}$ tensored with identities.
  A diagrammatic depiction of this decomposition is given in Figure~\ref{fig:Rmn}.
 
 \begin{figure}[htb]
 	\begin{center}
 		\includegraphics[width=1.3in]{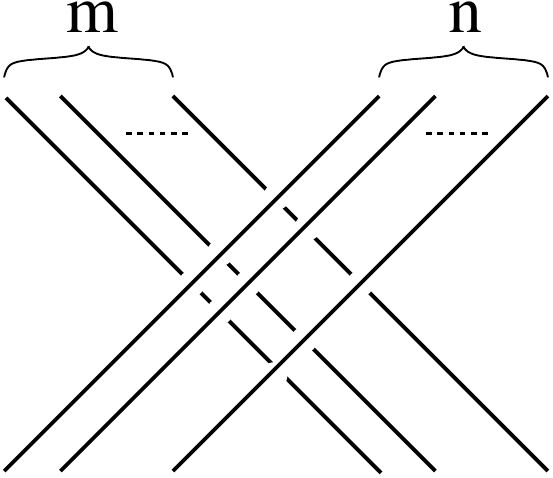}
 	\end{center}
 	\caption{Diagrammatic depiction of $R_{m,n}$, where each crossing of two strands corresponds to the R-matrix of Theorem~\ref{thm:main}}
 	\label{fig:Rmn}
 \end{figure}

We now show the naturality of the braiding. This accounts to showing that morphisms and braidings commute. However, from the
result of the preceding paragraph, by decomposing all the morphisms and the braidings into compositions of $R_{1,1}$ tensor the identity, we find that naturality is the same as the braid relation for $R_{1,1}$, which was proved in Theorem~\ref{thm:main}. 

We now proceed to proving the second part of the statement. Suppose first that $(X_1, H, \nu_1)$ and $(X_2, H,\nu_2)$ are equivalent categorical augmented racks. Then 
there exists 
 an isomorphism $f$ between $(X_1,H,\nu_1)$ and $(X_2,H,\nu_2)$. Observe that by the definition of isomorphism between categorical augmented racks, it follows that the associated SD structures are isomorphic. In particular, $f$ induces a morphism between the objects $X_1$ and $X_2$ of $\mathcal B(X_1, H, \nu_1)$ and $\mathcal B(X_2, H, \nu_2)$. 

Let us define a functor 
$F$ between $\mathcal B(X_1,H,\nu_1)$ and $\mathcal B(X_2, H, \nu_2)$. We set 
$F(X_1) = X_2$, and for all $n \in \mathbb N$ with $n\geq 2$ we define $F(X_1^{\otimes n}) = X_2^{\otimes n}$. When $n=0$ we set $F(\mathbb 1_{\mathbb k}) = \mathbb 1_{\mathbb k}$, the identity map over the one dimensional module $\mathbb k$. For a morphism $\phi \in {\rm Hom}(X_1^{\otimes n}, X_1^{\otimes n})$ we define $F(\phi)$ as follows. 
Denote the braiding in $\mathcal B(X_i, H, \nu_i)$ by $R^{(i)}$ for $i=1,2$. 
For the braiding $R^{(1)}: X_1 \otimes X_1 \rightarrow  X_1 \otimes X_1$,
 define 
$F(R^{(1)})= (f\otimes f)\circ R^{(1)} \circ (f^{-1}\otimes f^{-1}).$
Since $R^{(i)} (x \otimes y)=y^{(1)} \otimes ( x \cdot \nu_i (y^{(2)} )$, for $u\otimes  v=f(x) \otimes f(y)$ we have that 
\begin{eqnarray*}
 F(R^{(1)}) (u \otimes v) &=&
[ (f\otimes f)\circ R^{(1)} \circ (f^{-1}\otimes f^{-1}) ] (u \otimes v) \\
 &=&
[ (f\otimes f)\circ R^{(1)}] ( f^{-1} (u) \otimes f^{-1} (v) ) \\
&=&
(f \otimes f ) ( y^{(1)} \otimes ( x \cdot \nu_1(y^{(2)} ) ) \\
& = &
 v^{(1)} \otimes ( u \cdot \nu_2( v^{(2)} ) \\
&=& R^{(2)} (u \otimes v) ,
\end{eqnarray*}
so that $F (R^{(1)})=R^{(2)}$ as desired. 
If $\phi$ is a braiding
 $R_{s,t}^{(1)} : X^{\otimes s}\otimes X^{\otimes t} \rightarrow X^{\otimes t} \otimes X^{\otimes s}$ we set 
$$F(\phi) = F(R_{s,t}^{(1)} ) = (f^{\otimes t}\otimes f^{\otimes s})\circ R^{(1)}_{s,t}\circ ((f^{-1})^{\otimes s}\otimes (f^{-1})^{\otimes t}).$$
By computations similar to the case of $F (R^{(1)})=R^{(2)}$, we have that 
$F (R^{(1)}_{s,t})=R^{(2)}_{s,t}$.

If $\phi$ is a composition of morphisms of type 
$ {\mathbb 1}^{\otimes k_1} \otimes  R_{s,t}^{(1)}  \otimes  {\mathbb 1}^{\otimes k_2}$,
we define $F(\phi)$ as the composition of morphisms of type 
$ {\mathbb 1}^{\otimes k_1} \otimes  F(R_{s,t}^{(1)} )  \otimes  {\mathbb 1}^{\otimes k_2}$. 
From the definitions, we have
that the functor $F$ is a braided monoidal functor. By defining $F^{-1}$ analogously but by replacing $f$ by $f^{-1}$, we can construct an braided monoidal functor between $\mathcal B(X_2, H, \nu_2)$ and $\mathcal B(X_1, H, \nu_1)$ which is the inverse of $F$. This gives that the categories $\mathcal B(X_1 , H, \nu_1)$ and $\mathcal B(X_2, H, \nu_2)$ are  braided
 equivalent. 
 
 Now assume that the $H$-actions are faithful. Suppose that we have a braided monoidal equivalence $F$ between the two categories. Then there exists an isomoprhism between $X_1$ and $X_2$. Since $\mathcal B(X_1, H, \nu_1)$ and $\mathcal B(X_2, H, \nu_2)$ are categories of $H$-modules, the isomoprhism $f : X_1 \rightarrow X_2$ is $H$-equivariant. 
  Since
  $f$ commutes with the braidings, 
  we have that 
$$(f\otimes f)\circ R^{(1)} (x\otimes y) = R^{(2)} (f(x)\otimes f(y)).$$
For the first term we compute
\begin{eqnarray*}
		(f\otimes f)\circ R^{(1)} (x\otimes y) &=& (f\otimes f)(y^{(1)}\otimes q_1(x\otimes y^{(2)}))\\
		&=& (f\otimes f) (y^{(1)}\otimes x\cdot \nu_1(y^{(2)}))
\end{eqnarray*}
while for the second term we compute
\begin{eqnarray*}
			R^{(2)} (f(x)\otimes f(y)) &=& f(y^{(1)}) \otimes q_2(f(x)\otimes f(y^{(2)}))\\
			&=& f(y^{(1)})\otimes f(x)\cdot \nu_2(f(y^{(2)})),
\end{eqnarray*}
where we have used the fact that $f$ is a morphism of coalgebras. 
Applying $\epsilon_2\otimes \mathbb 1$ (where $\epsilon_2$ is the counit of $X_2$) to the equality $(f\otimes f)\circ R^{(1)} (x\otimes y) = R_{1,1}^{(2)} (f(x)\otimes f(y))$ we find that
$$
		f(x\cdot \nu_1(y)) = f(x) \cdot \nu_2(f(y)). 
$$
Using the fact that $f$ is $H$-equivariant and that $H$ acts on $X_2$ faithfully, we find that $\nu_1(y) = \nu_2(f(y))$. This shows that $f$ is a morphism of augmented racks. Proceeding similarly but for $f^{-1
}$ we see that $f$ is invertible through a morphism of augmented racks, completing the proof. 
 \end{proof}

%


\begin{thebibliography}{0}



%
		
	
	\bib{CCEScoalg}{article}{
		title={Cohomology of categorical self-distributivity},
		author={Carter, J. Scott},
		author={Crans, Alissa S.},
		author={Elhamdadi, Mohamed},
			author={Saito, Masahico},
		journal={J. Homotopy Relat. Struct.}, 
		volume={3},
		number={1},
		pages={13--63},
		year={2008}
	}

	\bib{CCESadj}{article}{ 
   AUTHOR = {Carter, J. Scott},
    author={Crans, Alissa S.},
    author={Elhamdadi, Mohamed},
             author={Saito, Masahico},
                  TITLE = {Cohomology of the adjoint of {H}opf algebras},
   JOURNAL = {J. Gen. Lie Theory Appl.},
    VOLUME = {2},
      YEAR = {2008},
    NUMBER = {1},
     PAGES = {19--34},
}
		\bib{CCES1}{article}{
      AUTHOR = {Carter, J. Scott},
    author={Crans, Alissa S.},
    author={Elhamdadi, Mohamed},
             author={Saito, Masahico},
     TITLE = {Homology theory for the set-theoretic {Y}ang-{B}axter equation
              and knot invariants from generalizations of quandles},
   JOURNAL = {Fund. Math.},
    VOLUME = {184},
      YEAR = {2004},
     PAGES = {31--54},
}
			

	\bib{CJKLS}{article}{
		title={Quandle cohomology and state-sum invariants of knotted curves and surfaces},
		author={Carter, J},
		author={Jelsovsky, Daniel},
		author={Kamada, Seiichi},
		author={Langford, Laurel},
		author={Saito, Masahico},
		journal={Trans. Amer. Math. Soc.}, 
		volume={355},
		number={10},
		pages={3947--3989},
		year={2003}
	}
	
	
       	
 
 
\bib{ChariPressley}{book}{
	title={A guide to quantum groups},
	author={Chari, Vyjayanthi},
	author={Pressley, Andrew},
	year={1995},
	publisher={Cambridge University Press}
}
 



	\bib{EGM}{article}{
		author={Elhamdadi, Mohamed},
		author={Green, Matthew},
		author={Makhlouf, Abdenacer},
		title={Ternary distributive structures and quandles},
		journal={Kyungpook Math. J.},
		volume={56},
		date={2016},
		number={1},
		pages={1--27},
	}
	
\bib{ESZheap}{article}{
	title={Heap and ternary self-distributive cohomology},
	author={Elhamdadi, Mohamed},
	author={Saito, Masahico},
	author={Zappala, Emanuele},
	journal={ Comm. Algebra }, 
	pages={2378--2401},
	year={2021},
	volume={49},
	number={6},
}

\bib{ESZcascade}{article}{
	author = {Elhamdadi, Mohamed},
	author={Saito, Masahico},
	author={Zappala, Emanuele},
	title = {Higher arity self-distributive operations in Cascades and their cohomology},
	journal = { J. Algebra Appl.}, 
	volume = {20},
	number = {7},
	year = {2021},
	pages={Paper No. 2150116, 33 pp.},
}

	\bib{FR}{article}{
	   author={Fenn, Roger},
	   author={Rourke, Colin},
	   title={Racks and Links in Codimension Two},
	   journal={ J. Knot Theory Ramifications }, 
	   volume={1},
	   number={4},
	   date={1992},
	   pages={343--406},
	}	



%



%


\bib{Joyce82}{article}{
    AUTHOR = {Joyce, David},
     TITLE = {A classifying invariant of knots, the knot quandle},
   JOURNAL = {J. Pure Appl. Algebra},
    VOLUME = {23},
      YEAR = {1982},
    NUMBER = {1},
     PAGES = {37--65},
}






		


%





\bib{Matveev82}{article}{
    AUTHOR = {Matveev, S. V.},
     TITLE = {Distributive groupoids in knot theory},
   JOURNAL = {Mat. Sb. (N.S.)},
    VOLUME = {119(161)},
      YEAR = {1982},
    NUMBER = {1},
     PAGES = {78--88, 160},
}

	
	
	\bib{SZframedlinks}{article}{
	title={Fundamental heap for  framed links and ribbon cocycle invariants},
	author={Saito, Masahico},
	author={Zappala, Emanuele},
	journal={J. Knot Theory  Ramifications}, 
	volume={32}, number={5}, 
	pages ={Paper No.~2350040},
	year={2023},

}

	\bib{SZbrfrob}{article}{
	title={Braided Frobenius Algebras from certain Hopf Algebras},
	author={Saito, Masahico},
	author={Zappala, Emanuele},
	journal={J. Algebra Appl.}, 
	volume={22},
	number={1},
	pages ={Paper No.~2350012},
	year={2023}
}
%
	


	


\end{thebibliography}
\end{document}